\documentclass[a4paper]{article}
\usepackage{amssymb}
\usepackage{amsmath}
\usepackage{amsthm}
\usepackage{color}
\usepackage{geometry}
\usepackage{epsf}
\usepackage{subfigure}
\usepackage{multirow}
\usepackage{mathrsfs}
\newtheorem{theorem}{Theorem}[section]
\newtheorem{lemma}[theorem]{Lemma}
\newtheorem{definition}[theorem]{Definition}
\newtheorem{remark}[theorem]{Remark}


\title{Existence results for some nonlinear elliptic systems on graphs}

\begin{document}

\date{2023.8.17}

\author{Shoudong Man\footnote{manshoudong@163.com}\\
\it College of science and technology, Tianjin University of Finance and Economics,\\ Zhujiang Rd, Tianjin 300222,  China}

\maketitle

\begin{abstract}
In this paper, several nonlinear elliptic systems are investigated on graphs.
One type of the sobolev embedding theorem and a new version of the strong maximum principle are established.
Then, by using the variational method, the existence of different types of solutions to some elliptic systems is confirmed.
Such problems extend the existence results on closed Riemann surface to graphs
and extend the existence results for one single equation on graphs [A. Grigor'yan, Y. Lin, Y. Yang, J. Differential Equations, 2016]
to nonlinear elliptic systems on graphs.
Such problems can also be viewed as one type of discrete version of the elliptic systems on Euclidean space and Riemannian manifold.



\end{abstract}

\section{Introduction}

 Recently, the study of discrete Laplacians and various equations on graphs have attracted much attention.
 About six years ago, A. Grigor'yan, Y. Lin and Y. Yang systematically raised and studied Yamabe equations, Kazdan-Warner equations and Schr\"odinger equations on graphs \cite{GLA,AGY,GAL}. They first
established the Sobolev spaces and the functional framework on graphs, and as a consequence variational methods are applied to solve partial differential equations on graphs. In particular, on a locally finite graph $G = (V, E)$, they proved that there exists a positive solution to
\begin{equation}\label{1.3}
\left \{
\begin{array}{lcr}
 -\Delta u-\alpha u=|u|^{p-2}u &{\rm in}& \Omega^{\circ}\\
  u=0 &{\rm in}& \partial \Omega\\
\end{array}
\right.
\end{equation}
for any $p>2$ and finite $\Omega \subset V$. On a finite graph $G = (V, E)$, they proved that there exists a positive solution to
\begin{equation}\label{1.4}
\left \{
\begin{array}{lcr}
 -\Delta u+ h(x)u=|u|^{p-2}u &{\rm in}& V\\
  u>0&{\rm in}& V\\
\end{array}
\right.
\end{equation}
for any $p>2$ and $h(x)>0$.
Moreover, they obtained that there exists a strictly positive solution to
\begin{equation}\label{1.5}
\Delta u+ h(x)u=f(x,u)
\end{equation}
for certain positive function $h(x)$ and certain nonlinearity $f(x,u)$. They also considered the existence of positive solutions of the
perturbed equation
\begin{equation}\label{1.6}
\Delta u+ hu=f(x,u)+\epsilon g.
\end{equation}
In \cite{GAL2}, Y. Lin and Y. Yang introduced a heat flow method of solving the Kazdan-Warner problem.
In \cite{GAL1}, Y. Lin and Y. Yang studied the following mean field equation
\begin{equation}\label{buchong}
\left \{
\begin{array}{lcr}
 -\Delta u+ h(x)u=\frac{ge^{u}}{\int_{V}ge^{u}d\mu}-f(x)~~in~V\\
  u\in \mathfrak{H}\cap L^{\infty}(V)
\end{array}
\right.
\end{equation}
on locally finite graphs and obtained a global solution to the equation.
In \cite{shoudong1}, by establishing an eigenfunction space $\mathcal{H}$, we proved that there exists a nontrivial solution $u\in\mathcal{H}$ to
one type of nonlinear Schr\"odinger equation on finite graphs. In \cite{shoudong2}, we considered a class of quasilinear elliptic equation with
indefinite weights on a locally finite graph, and obtained the existence of a positive solution by using variational methods.
For more results about differential equations on graphs, we refer the readers
to \cite{HUABIN,BIN1,BIN2,GLA,GAL,GAL4,BOBO,GAL1,GAL2,GAL3,LIN,XA,NZL,ZD}.

Nonlinear elliptic systems have extensive practical applications in dynamics, biology, engineering, physics and the other sciences.
On the Euclidean space and Riemannian manifold, the investigation of nonlinear differential equation systems
have deserved a great deal of interest. For examples, the readers are referred to \cite{Correa,Guofeng,Costa,Finueirdo,Scientia,Huei,Robertson, Dongdong}
and the references therein.

In this paper, we study the problems of nonlinear elliptic systems by replacing a Riemann surface or an Euclidean domain with a graph.
We extend the existence results for the Toda system on closed Riemann surface to graphs.
We also extend the existence results for one single equation such as (\ref{1.3})-(\ref{buchong}) to nonlinear elliptic systems on graphs.
On the Euclidean space, say $\Omega\subset \mathbb{R}^{n}$, the Sobolev embedding theorem reads
\begin{equation*}
W^{1,q}_{0}(\Omega)\hookrightarrow
\left \{
\begin{array}{lcr}
 L^{p}(\Omega) \ \ for \ \  q\leq p\leq q^{*}, \ \ when \ \ n>q,\\
 L^{p}(\Omega) \ \ for \ \  q\leq p\leq +\infty, \ \ when \ \ n=q, \\
  C^{1-\frac{n}{q}}(\bar{\Omega}), \ \  when \ \ n<q .\\
\end{array}
\right.
\end{equation*}
However, the Sobolev embedding theorems on graphs are quite different. Letting $G_{0}=(V_{0},E_{0})$ be a locally finite graph and $\Omega$ be a bounded domain of $V_{0}$ with $\Omega^{\circ}\neq\emptyset$, the Sobolev embedding theorem on $G_{0}$ reads
\begin{equation*}
W^{1,s}_{0}(\Omega)\hookrightarrow
 L^{\gamma}(\Omega) \ \ for \ \  s>1\ and\ \ 1\leq \gamma\leq+\infty.
\end{equation*}
While on a finite graph $G=(V,E)$, the Sobolev embedding theorem reads
\begin{equation*}
W^{1,s}(V)\hookrightarrow
 L^{\gamma}(V) \ \ for \ \  s>1\ and\ \ 1\leq \gamma\leq+\infty.
\end{equation*}
Hence, it allows us to assume different growth conditions on the nonlinear term $f$ and $g$ on graphs, and our results are different from that of \cite{Scientia,Huei} in the case $p=2^{*}$.

On graphs, in 2022, L. Zhao \cite{L.Zhao} and M. Shao\cite{M.Shao} systematically raised and studied the Laplacian system and the p-Laplacian system respectively. However, up to now, there are very few results about differential equation systems on graphs in this connection.
In this paper, we investigate the global existence results and the local existence results for the Toda system and the other nonlinear elliptic
systems involving the Laplacian, the p-Laplacian and the poly-Laplacian.
Particularly, we establish some eigenfunction spaces, and prove that there exist some solutions which belong to these eigenfunction spaces for
the Toda system.

This paper is organized as follows: In section 2, we give some notations and settings on graphs. In section 3, we state our
main theorems. In section 4, we give some preliminary results. In section 5-7, we prove the main theorems in section 3.

\section{Notations and settings}\label{sub001}
Let $G = (V, E)$ be a finite or locally finite graph, where $V$ denotes the vertex set and $E$ denotes the edge set. The degree of vertex $x$, denoted by $\psi(x)$, is the number of edges connected to $x$.
If for every vertex $x$ of $V$, the number of edges connected to $x$ is finite, we say that $G$ is a locally finite graph. We denote $x\sim y$ if vertex $x$ is adjacent to vertex $y$. We use $(x,y)$ to denote an edge in $E$ connecting vertices $x$ and $y$.
Let $\omega_{xy}=\omega_{yx}>0$ where $\omega_{xy}$ is the edge weight. The finite measure $\psi(x)=\sum_{y\sim x}\omega_{xy}$.
A graph $G$ is called connected if for any vertices $x,y\in V$, there exists a sequence $\{x_{i}\}_{i=0}^{n}$ satisfying
$x=x_{0}\sim x_{1}\sim x_{2}\sim \cdot\cdot\cdot \sim x_{n}=y$.

From \cite{AGY}, for any function $u:V\rightarrow \mathbb{R}$, the $\psi(x)$-Laplacian of $u$ is defined as
\begin{equation*}\label{86.1}
\Delta u(x)=\frac{1}{\psi(x)}\sum_{y\sim x}\omega_{xy}\big{(}u(y)-u(x)\big{)}.
\end{equation*}
The associated gradient form reads
 \begin{align*}
\Gamma(u,v)(x)=&\frac{1}{2}\big{\{}\Delta\big{(}u(x)v(x)\big{)}-u(x)\Delta v(x)-v(x)\Delta u(x)\big{\}} \\
=&\frac{1}{2\psi(x)}\sum_{y\sim x} \omega_{xy}\big{(}u(y)-u(x)\big{)}\big{(}v(y)-v(x)\big{)}.
\end{align*}
The length of the gradient for $u$ is
\begin{equation*}\label{86.5}
|\nabla u|(x)=\sqrt{\Gamma(u,u)(x)}=\Big{(}\frac{1}{2\psi(x)}\sum_{y\sim x} \omega_{xy}\big{(}u(y)-u(x)\big{)}^{2}\Big{)}^{1/2}.
\end{equation*}
Similarly, according to \cite{AGY}, the length of the m-order gradient of $u$ is
\begin{equation}\label{86.6}
|\nabla^{m} u|(x)=
\left \{
\begin{array}{lcr}
|\nabla \Delta^{\frac{m-1}{2}} u|, \ \ {\rm when} \ m \ {\rm is} \ {\rm odd},\\
|\Delta^{\frac{m}{2}} u|, \ \ {\rm when} \ m \ {\rm is} \ {\rm even},\\
\end{array}
\right.
\end{equation}
where $|\Delta^{\frac{m}{2}} u|$  denotes the usual
absolute of the function $\Delta^{\frac{m}{2}} u$.
To compare with the Euclidean setting, we denote the integral of a function $u: V \rightarrow
\mathbb{R}$ as
\begin{equation*}\label{86.9}
\int_{V}ud\psi=\sum_{x\in V}\psi(x)u(x).
\end{equation*}

By \cite{Fan}, all eigenvalues of the Laplacian $-\Delta$ on $G=(V,E)$ are non-negative. The minimum positive eigenvalue also called the first positive eigenvalue can be defined as
\begin{align}\label{1.12}
\lambda_{1}=\inf_{u\perp T\mathbf{1}} \frac{\sum_{x,y\in V, y\sim x}\big{(}u(y)-u(x)\big{)}^{2}\omega_{xy}}{\sum_{x\in V}u^{2}(x)\psi(x)}
=\inf_{u\perp T\mathbf{1}}\frac{\int_{V}|\nabla u|^{2}d\psi}{\int_{V}u^{2}d\psi},
\end{align}
where $T\mathbf{1}$ is a vector whose each element is the degree of the corresponding vertex.
The nontrivial function $u$ that achieves (\ref{1.12}) is called a
harmonic eigenfunction of $-\Delta$ on $G$ with eigenvalue $\lambda_{1}$.
For more details, we refer the readers to \cite{Fan}.

By Lemma 1.10 in \cite{Fan},
if $u(x)$ is a harmonic eigenfunction achieving $\lambda_{1}$ in (\ref{1.12}),
then, for any vertex $x\in V$, we have
\begin{align}\label{1.13}
-\Delta u(x)=\frac{1}{\psi(x)}\sum_{y\sim x}\omega_{xy}\big{(}u(x)-u(y)\big{)}=\lambda_{1} u(x).
\end{align}
If $v(x)$ is also a harmonic eigenfunction with $\lambda_{1}$, it is easy to see that $u(x)+v(x)$ and $ku(x)$ are also harmonic eigenfunctions with $\lambda_{1}$, where $k$ is a constant. If $u(x)\equiv 0$, by (\ref{1.13}), $u(x)$ can be still seen as a harmonic eigenfunction with $\lambda_{1}$.

By \cite{AGY}, in the distributional sense, the $p$-Laplacian of $u:V\rightarrow \mathbb{R}$, namely $\Delta_{p}u$, is defined as
\begin{align}\label{Lap}
\Delta_{p} u(x)=\frac{1}{2\psi(x)}\sum_{y\sim x}\omega_{xy}\Big{(}|\nabla u|^{p-2}(y)+|\nabla u|^{p-2}(x)\Big{)}\Big{(}u(y)-u(x)\Big{)}.
\end{align}
Let $u:V\rightarrow \mathbb{R}$ and $\mathcal{L}_{m,p}u$ can be defined in the distributional sense:
for any function $\phi$, there holds
\begin{equation}\label{88.6}
\int_{V}(\mathcal{L}_{m,p}u)\phi d\psi=
\left \{
\begin{array}{lcr}
\int_{V}|\nabla^{m}u|^{p-2} \Gamma(\Delta^{\frac{m-1}{2}}u, \Delta^{\frac{m-1}{2}}\phi)d\psi, \ \ {\rm when} \ m \ {\rm is} \ {\rm odd},\\
\int_{V}|\nabla^{m}u|^{p-2} \Delta^{\frac{m}{2}}u\Delta^{\frac{m}{2}}\phi d\psi, \ \ {\rm when} \ m \ {\rm is} \ {\rm even}.\notag
\end{array}
\right.
\end{equation}
In particular, the poly-Laplacian $\mathcal{L}_{m,2}u=(-\Delta)^{2}u$, while $\mathcal{L}_{1,2}u=-\Delta u$.
For any $s>0$,  $L^{s}(V)$ denotes the Banach space with the norm
\begin{equation}\label{866.11}
||u||_{L^{s}(V)}=(\int_{V}|u|^{s}d\psi)^{1/s},
\end{equation}
and $L^{\infty}(V)$ means
\begin{equation}\label{866.121212}
||u||_{L^{\infty}(V)}=\sup_{x\in V}|u(x)|<\infty.
\end{equation}
For any $p>1$ and $m\geq 1$,  $W^{m,p}(V)$ is defined as a space of all functions $u: V \rightarrow \mathbb{R}$ under the norm
\begin{equation}\label{866.10}
||u||_{W^{m,p}(V)}=\Big{(}\int_{V}\big{(}|\nabla^{m}u|^{p}+h(x)|u|^{p}\big{)}d\psi\Big{)}^{\frac{1}{p}},
\end{equation}
where $h(x)>0$ for all $x\in V$. When $p=2$ and $h(x)\equiv 1$ for all $x\in V$ in (\ref{866.10}), for any integer $m\geq 1$, we define a new subspace of $W^{m,2}(V)$ denoted by $\mathcal{H}^{m,2}(V)$ that satisfies
\begin{align}\label{tod5}
\mathcal{H}^{m,2}(V)=\{u\in W^{m,2}(V)\mid -\Delta u=\lambda_{1} u\}.
\end{align}
The norm on $\mathcal{H}^{m,2}(V)$ is defined as
\begin{equation*}
||u||_{\mathcal{H}^{m,2}(V)}=\Big{(}\int_{V}\big{(}|\nabla^{m}u|^{2}+|u|^{2}\big{)}d\psi\Big{)}^{\frac{1}{2}}.
\end{equation*}
Let $\Omega \subset V$ and $\Omega$ be a bounded domain. The boundary of $\Omega$ is denoted by $\partial\Omega$,
where $\partial\Omega=\{x\in \Omega: \exists y\not\in \Omega ~~{\rm such} ~~{\rm that} ~~(x,y)\in E \}$. Moreover, the interior of $\Omega$ is denoted by
$\Omega^{\circ}$ that satisfies $\Omega^{\circ}=\Omega\setminus \partial\Omega$. For any $p>1$ and integer $m\geq 1$, $W^{m,p}(\Omega)$ is defined as a space of all functions $u: \Omega \rightarrow \mathbb{R}$ that satisfies
\begin{equation}\label{86.10}
||u||_{W^{m,p}(\Omega)}=\Big{(}\sum^{m}_{k=0}\int_{\Omega}|\nabla^{k}u|^{p}d\psi\Big{)}^{\frac{1}{p}}<+\infty.
\end{equation}
Denote $C^{m}_{0}(\Omega)$ as a set of all functions $u : \Omega \rightarrow \mathbb{R}$ with $u = |\nabla u| = \cdot\cdot\cdot = |\nabla^{m-1} u| = 0$ on $\partial\Omega$,
and $W_{0}^{m,p}$ as the completion of $C^{m}_{0}(\Omega)$ under the norm (\ref{86.10}). It is easy to see that
\begin{equation}\label{86.11}
||u||_{W_{0}^{m,p}(\Omega)}=\Big{(}\int_{\Omega}|\nabla^{m}u|^{p}d\psi\Big{)}^{\frac{1}{p}}.
\end{equation}
 For any integer $m\geq 1$, we define a new subspace of $W_{0}^{m,2}(\Omega)$ denoted by $\mathcal{H}_{0}^{m,2}(\Omega)$ that satisfies
\begin{align}\label{tod5}
\mathcal{H}_{0}^{m,2}(\Omega)=\{u\in W_{0}^{m,2}(\Omega)\mid -\Delta u=\lambda_{1} u\}.
\end{align}
The norm on $\mathcal{H}_{0}^{m,2}(\Omega)$ is defined as
\begin{equation*}
||u||_{\mathcal{H}_{0}^{m,2}(\Omega)}=\Big{(}\int_{\Omega}\big{(}|\nabla^{m}u|^{2}\big{)}d\psi\Big{)}^{\frac{1}{2}}.
\end{equation*}
The $\mathcal{L}_{m,p}u$ with $u:\Omega\rightarrow \mathbb{R}$ can be defined as follows:
for any $\phi$ with $\phi=|\nabla \phi|=\cdot\cdot\cdot=|\nabla^{m-1} \phi|=0$ on $\partial\Omega$, there holds
\begin{equation}\label{pol}
\int_{\Omega}(\mathcal{L}_{m,p}u)\phi d\psi=
\left \{
\begin{array}{lcr}
\int_{\Omega}|\nabla^{m}u|^{p-2} \Gamma(\Delta^{\frac{m-1}{2}}u, \Delta^{\frac{m-1}{2}}\phi)d\psi, \ \ {\rm when} \ m \ {\rm is} \ {\rm odd},\\
\int_{\Omega}|\nabla^{m}u|^{p-2} \Delta^{\frac{m}{2}}u\Delta^{\frac{m}{2}}\phi d\psi,  \ \ \ \ \ \ {\rm when} \  m \ {\rm is} \ {\rm even}.\\
\end{array}
\right.
\end{equation}

By \cite{GAL1}, on graph $G =(V,E)$, for any fixed $O\in V$, the distance between vertex $x$ and $O$, denoted by $\rho(x)=\rho(x,O)$, is the minimum number of edges connecting them.

\section{Main theorems about the existence results on graphs}

In this section, we investigate the existence results in three cases, that is, the global existence results for the Toda system,
the local existence results and the global existence results for the other nonlinear elliptic
systems involving the Laplacian, the p-Laplacian and the
poly-Laplacian.

\subsection{Global existence results for the Toda system on graphs}\label{tezhengzhi}
By \cite{Jost,chunqin}, the Toda system is the system of $2\times 2$ Liouville equations on $\Sigma$

\begin{equation}\label{tod5a1}
\left \{
\begin{array}{lcr}
-\Delta u_{1}=\varphi_{1}(\frac{e^{2u_{1}-u_{2}}}{\int_{\Sigma}e^{2u_{1}-u_{2}}dv}-1), \\
-\Delta u_{2}=\varphi_{2}(\frac{e^{-u_{1}+2u_{2}}}{\int_{\Sigma}e^{-u_{1}+2u_{2}}dv}-1), \\
\int_{\Sigma}u_{1}dv=0,\\
\int_{\Sigma}u_{2}dv=0\\
 \end{array}
\right.
\end{equation}
where $\Sigma$ is a closed Riemann surface. In \cite{Jost}, the authors proved that there exists a solution to the system (\ref{tod5a1}).
When $u_{1}=u_{2}$, each of the first two equations in this system can be viewed as the same as the equation in (\ref{buchong}) in the case $f(x)\equiv 1$
and $h(x)\equiv 0$.

Here we investigate the global existence results for system (\ref{tod5a1}) on finite and locally finite graphs respectively.

Let $G=(V,E)$ denote a connected finite or locally finite graph.
For any integer $m,n\geq 1$, let $\mathcal{H}=\mathcal{H}^{m,2}(V)\times \mathcal{H}^{n,2}(V)=\{(u,v)\in W^{m,2}(V)\times W^{n,2}(V)\mid -\Delta u=\lambda_{1} u, -\Delta v=\lambda_{1} v \} $ be the space with the norm
\begin{align*}
||(u,v)||_{\mathcal{H}}=\max\Big{\{} ||u||_{\mathcal{H}^{m,2}(V)}, ||v||_{\mathcal{H}^{n,2}(V)}\Big{\}}.
\end{align*}
From \cite{YLS} (see page 353) or by the formula of integration by parts on graphs, if $u$ and $v$ are harmonic eigenfunctions of $-\Delta$ on $G$ with eigenvalue $\lambda_{1}$,
that is $(u,v)\in \mathcal{H}$, we have that
\begin{align}\label{toooo6}
\int_{V}ud\psi=0~~ and ~~\int_{V}vd\psi=0.
\end{align}
Thus, analogous to (\ref{tod5a1}), when $u$ and $v$ are harmonic eigenfunctions of $-\Delta$ on $G$ with eigenvalue $\lambda_{1}$, it is equal to consider the following system on graphs
\begin{equation}\label{EEERRR}
\left \{
\begin{array}{lcr}
 \mathcal{L}_{m,2} u=\varphi_{1}(\frac{e^{2u-v}}{\int_{V}e^{2u-v}d\psi}-1) ~ &{\rm in}& V,\\
\mathcal{L}_{n,2} v=\varphi_{2}(\frac{e^{-u+2v}}{\int_{V}e^{-u+2v}d\psi}-1)~  &{\rm in}& V,
 \end{array}
\right.
\end{equation}
where $\varphi_{1}$ and $\varphi_{2}$ are two positive constants.

\begin{theorem}\label{EEET23}
Suppose that $G =(V,E)$ is a finite graph. For all $ x\in V$, assume that the distance function $\rho(x) = \rho(x, O) \in L^{q}(V)$ for some $q > 0$ and $O \in V$, and there exists a constant $\omega_{0}>0$ such that $\omega_{xy}\geq\omega_{0}$ for all $y\sim x$. Let $u$ and $v$ be two harmonic eigenfunctions of $-\Delta$ on $G$ with eigenvalue $\lambda_{1}$.
Then for any integer $m,n\geq 1$, there exists a solution $(u,v)\in \mathcal{H}$ to the problem $(\ref{EEERRR})$.
\end{theorem}

Now, we consider the problem (\ref{EEERRR}) on a locally finite graph.
Firstly, for any integer $k\geq 1$, we define a sequence of balls centered at $O$ with radius $k$ by
\begin{equation*}\label{866.13}
V_{k}=V_{k}(O)=\{x\in V\mid \rho(x)<k\}.
\end{equation*}
The boundary of $V_{k}$ can be defined as
\begin{equation*}\label{866.17}
\partial V_{k}=\{x\in V\mid \rho(x)=k\}.
\end{equation*}
Our next result is as follows:
\begin{theorem}\label{EEET238}
Suppose that $G =(V,E)$ is a locally finite graph. For all $ x\in V$, assume that the distance function $\rho(x) = \rho(x, O) \in L^{q}(V)$ for some $q > 0$ and $O \in V$, and there exists a constant $\omega_{0}>0$ such that $\omega_{xy}\geq\omega_{0}$ for all $y\sim x$.
Let $u$ and $v$ be two harmonic eigenfunctions of $-\Delta$ on $G$ with eigenvalue $\lambda_{1}$.
Then for any integer $m,n\geq 1$, there exists a solution $(u,v)\in \mathcal{H}$ to the problem $(\ref{EEERRR})$.
\end{theorem}

\begin{remark}
 The existence results of the solutions to the Toda system $(\ref{tod5a1})$ on connected finite
 and locally finite graphs are quite different from that on compact Riemann surfaces.
The Toda system $(\ref{tod5a1})$ always has a minimum solution on connected finite
(or locally finite) graphs if $\varphi_{1},\varphi_{2}>0$, while on compact Riemann surface,
the Toda system $(\ref{tod5a1})$ does not have minimum solution when $\varphi_{1}>4\pi$ or $\varphi_{2}>4\pi$.
 We also get that the minimum solution to $(\ref{tod5a1})$ is $(u,v)\in \mathcal{H}$, where $u$ and $v$ are harmonic eigenfunctions of $-\Delta$ on $G$ with eigenvalue $\lambda_{1}$,
\end{remark}

\subsection{Local existence results on locally finite graphs}\label{section1}

Now, we consider some nonlinear elliptic systems on locally finite graphs.
 We extend the existence results for one single equation on graphs such as (\ref{1.3}) and (\ref{1.4}) to nonlinear elliptic systems on graphs. Also such problems can be viewed as one type of discrete version of the elliptic systems on Euclidean space and Riemannian manifold.

Let $G = (V,E)$ be a locally finite graph and $\Omega$ be a bounded domain of $V$ with $\Omega^{\circ}\neq\emptyset$.
The first system we consider is the following semilinear elliptic system
\begin{equation}\label{E1}
\left \{
\begin{array}{lcr}
 -\Delta u=\frac{p}{p+q}|u|^{p-2}u|v|^{q} &{\rm in}&\Omega^{\circ},\\
-\Delta v=\frac{q}{p+q}|u|^{p}|v|^{q-2}v  &{\rm in}&\Omega^{\circ},\\
u>0 ~and ~v>0 &{\rm in}&\Omega^{\circ},\\
u=v=0 &{\rm on}&\partial\Omega,
 \end{array}
\right.
\end{equation}
where $p, q>2$.

Let $\mathbb{H}= W^{1,2}_{0}(\Omega)\times W^{1,2}_{0}(\Omega)$ be the space with the norm
\begin{align*}
||(u,v)||_{\mathbb{H}}= \Big{(}\int_{\Omega}(|\nabla u|^{2}+|\nabla v|^{2})d\psi\Big{)}^{\frac{1}{2}}.
\end{align*}
Associated with the problem (\ref{E1}), we consider the $C^{1}$-functional $J_{3}(u,v):\mathbb{H}\rightarrow \mathbb{R}$ for $(u,v)\in\mathbb{H}$, where
$$ J_{3}(u,v)=\frac{1}{2}\int_{\Omega}(|\nabla u|^{2}+|\nabla v|^{2})d\psi-\frac{1}{p+q}\int_{\Omega}|u|^{p}|v|^{q}d\psi.$$
The weak solution $(u,v)\in\mathbb{H}$ to problem (\ref{E1}) is the critical point of the functional $J_{3}(u,v)$, that is,
$(u,v)\in\mathbb{H}$ satisfies
\begin{align*}
\int_{\Omega}&(\nabla u \nabla\varphi_{01}+\nabla v \nabla\varphi_{02})d\psi-\frac{p}{p+q}\int_{\Omega}|u|^{p-2}u|v|^{q}\varphi_{01}d\psi \\
&-\frac{q}{p+q}\int_{\Omega}|u|^{p}|v|^{q-2}v\varphi_{02}d\psi=0
\end{align*}
for any $(\varphi_{01},\varphi_{02})\in\mathbb{H}$.

\begin{theorem}\label{T1}
Let $G = (V,E)$ be a locally finite graph and $\Omega$ be a bounded domain of $V$ with $\Omega^{\circ}\neq\emptyset$.
Then for any $p,q > 2$, there exists a solution $(u,v)\in \mathbb{H}$ to the problem $(\ref{E1})$.
\end{theorem}

Let $\mathbb{W}= W^{1,p}_{0}(\Omega)\times W^{1,q}_{0}(\Omega)$ be the space with the norm
\begin{align*}
||(u,v)||_{\mathbb{W}}=\max\Big{\{} ||u||_{W^{1,p}_{0}(\Omega)}, ||v||_{W^{1,q}_{0}(\Omega)}\Big{\}},
\end{align*}
and we consider the following system:
\begin{equation}\label{E2}
\left \{
\begin{array}{lcr}
 -\Delta_{p} u=\lambda_{0}|u|^{\alpha-1}u|v|^{\beta+1} &{\rm in}&\Omega^{\circ},\\
-\Delta_{q} v=\lambda_{0}|u|^{\alpha+1}|v|^{\beta-1}v &{\rm in}&\Omega^{\circ},\\
u>0 ~and ~v>0 &{\rm in}&\Omega^{\circ},\\
u=v=0 &{\rm on}&\partial\Omega.
 \end{array}
\right.
\end{equation}
\begin{theorem}\label{T2}
Let $G =(V,E)$ be a locally finite graph and $\Omega$ be a bounded domain of $V$ such that $\Omega^{\circ}\neq \emptyset$.
Then for any $\alpha+1>p$, $\beta+1>q$ and $\lambda_{0}>0$, there exists a solution $(u,v)\in \mathbb{W}$ to the problem $(\ref{E2})$.
\end{theorem}

Now, we consider the following equation system
\begin{equation}\label{E13}
\left \{
\begin{array}{lcr}
 \mathcal{L}_{m,p} u=\lambda \omega(x)|u|^{p-2}u+\frac{\alpha+1}{\alpha+\beta+2}|u|^{\alpha-1}u|v|^{\beta+1} &{\rm in}&\Omega^{\circ},\\
\mathcal{L}_{n,q} v=\vartheta\psi(x)|v|^{q-2}v+\frac{\beta+1}{\alpha+\beta+2}|u|^{\alpha+1}|v|^{\beta-1}v &{\rm in}& \Omega^{\circ},\\
| \nabla^{i} u|=0, ~~0\leq i\leq m-1, &{\rm on}& \partial\Omega,\\
| \nabla^{j} v|=0 ,~~0\leq j\leq n-1, &{\rm on}& \partial\Omega,
 \end{array}
\right.
\end{equation}
where $\alpha+1>p>1$, $\beta+1>q>1$, $m,n\geq 2$, $\omega(x)$ and $\psi(x)$ are given functions
which may change sign, and $\lambda$ and $\vartheta$ are the eigenvalue parameters.

Let $\mathcal{Z}=W^{m,p}_{0}(\Omega)\times W^{n,q}_{0}(\Omega)$ be the space with the norm
\begin{align*}
||(u,v)||_{\mathcal{Z}}=\max\Big{\{} ||u||_{W^{m,p}_{0}(\Omega)}, ||v||_{W^{n,q}_{0}(\Omega)}\Big{\}}.
\end{align*}
We define
\begin{align}
\lambda_{1}=\inf_{(u,v)\in \mathcal{Z}\backslash \{(0,0)\}}\frac{\int_{\Omega}|\nabla^{m} u|^{p}d\psi}{\int_{\Omega}\omega(x)|u|^{p}d\psi},\label{canshu111}\\
\vartheta_{1}=\inf_{(u,v)\in \mathcal{Z}\backslash \{(0,0)\}}\frac{\int_{\Omega}|\nabla^{n} v|^{q}d\psi}{\int_{\Omega}\psi(x)|v|^{q}d\psi}.\label{canshu122}
\end{align}
From \cite{Cuesta}, it is not difficult to get that $\lambda_{1},\vartheta_{1}>0$.

 Our next result reads
\begin{theorem}\label{T13}
Let $G =(V,E)$ be a locally finite graph and $\Omega \subset V$ be a bounded domain with $\Omega^{\circ}\neq \emptyset$.
Then for any $\lambda\in (0,\lambda_{1})$ and $\vartheta\in (0,\vartheta_{1})$, there exists a solution $(u,v)\in \mathcal{Z}\backslash \{(0,0)\}$ to the problem $(\ref{E13})$.
\end{theorem}

\subsection{Global existence results on finite graphs}

Now, analogous to Theorem \ref{T1}-\ref{T13}, we consider the global existence results for some elliptic systems on finite graphs $G=(V,E)$.
We extend the existence results for one single equation on graphs such as (\ref{1.5}) and (\ref{1.6}) to nonlinear elliptic systems on graphs.

Let $\mathbb{L}= W^{1,2}(V)\times W^{1,2}(V)$ be the space with the norm
\begin{align*}
||(u,v)||_{\mathbb{L}}=\Big{(}\int_{V}\big{(}|\nabla u|^{2}+|\nabla v|^{2}+h(x)|u|^{2}+h(x)|v|^{2}\big{)}d\psi\Big{)}^{\frac{1}{2}}.
\end{align*}
We have the following Theorem \ref{T3} analogous to Theorem \ref{T1}.
\begin{theorem}\label{T3}
Let $G =(V,E)$ be a finite graph and $h(x)>0$ for all $x\in V$.
 Suppose that $f(x,u): V\times \mathbb{R}\rightarrow \mathbb{R}$ and $g(x,v): V\times \mathbb{R}\rightarrow \mathbb{R}$ satisfy the following hypotheses:\\
\noindent$(H1)$ For all $x\in V$, $f(x,t)$ and $g(x,t)$ are
 continuous in $t\in \mathbb{R}$, $f(x,0)=g(x,0)=0$, \\
\indent \indent and $f(x,t),g(x,t)\geq 0$ for all $(x,t)\in V\times [0,+\infty)$;\\
\noindent$(H2)$ There exists a positive number $s>1$, such that $f(x,u)=o(|u|^{s})$ and \\
\indent \indent $g(x,v)=o(|v|^{s})$ uniformly for $x\in V$ as $|u|,|v|\rightarrow \infty$;\\
\noindent$(H3)$ $f(x,u)=o(|u|)$ and $g(x,v)=o(|v|)$ uniformly for $x\in V$ as $|u|,|v|\rightarrow 0$;\\
\noindent$(H4)$ There exists $\theta>2$ and $s_{0}>0$ such that if $s>s_{0}$, there hold\\
\indent \indent $\theta F(x,s)<f(x,s)s$ and $\theta G(x,s)<g(x,s)s$ for any $x\in V$,\\
\indent \indent  where $F(x,s)=\int_{0}^{s}f(x,t)dt$ and $G(x,s)=\int_{0}^{s}g(x,t)dt$.\\
Then, there exists a solution $(u,v)\in \mathbb{L}$ to the following problem
\begin{equation}\label{E3}
\left \{
\begin{array}{lcr}
 -\Delta u+h(x)u=g(x,v) &{\rm in}& V,\\
-\Delta v+h(x)v=f(x,u) &{\rm in}& V,\\
u>0 ~and ~v>0 &{\rm in}& V.
 \end{array}
\right.
\end{equation}
\end{theorem}

\begin{remark}
 $(u,v)$ is called a weak solution to $(\ref{E3})$, if $(u,v)\in \mathbb{L}$ satisfies
$$\int_{V}(\nabla u \nabla \phi+h(x)u\phi)d\psi+\int_{V}(\nabla v\nabla\nu+h(x)u\nu)d\psi-\int_{V}g(x,v)\phi d\psi-\int_{V}f(x,u)\nu d\psi=0$$
for all $(\phi,\nu)\in \mathbb{L}$.
\end{remark}

Let $\mathbb{X}= W^{1,p}(V)\times W^{1,q}(V)$ be the space with the norm
\begin{align*}
||(u,v)||_{\mathbb{X}}=\max\Big{\{} ||u||_{W^{1,p}(V)}, ||v||_{W^{1,q}(V)}\Big{\}}.
\end{align*}
Analogous to Theorem \ref{T2}, we have
\begin{theorem}\label{T4}
Let $G =(V,E)$ be a finite graph and $h(x)>0$ for all $x\in V$.
 Suppose that $f(x,u,v)$ and $g(x,u,v)$ satisfy the following hypotheses:\\
\noindent$(H1)$ There exists $F(x,u,v): V\times \mathbb{R}\times\mathbb{R}\rightarrow \mathbb{R}$,
 where $F_{u}(x,u,v)=f(x,u,v)$, \\
 \indent \indent  $F_{v}(x,u,v)=g(x,u,v)$, $F(x,0,0)=f(x,0,0)=g(x,0,0)=0$ for all\\
 \indent \indent   $x\in V$; $f(x,u,v)\geq 0$ and $g(x,u,v)\geq 0 $ for all $u,v\geq 0$. \\
\noindent$(H2)$ There exist $r>p,s>q$ and $C>0$ such that $F(x,u,v)\leq C(1+|u|^{r}+|v|^{s})$ \\
  \indent \indent for all $x\in V$;\\
\noindent$(H3)$ There exist $\bar{r}>p,\bar{s}>q,C_{0}>0$ and $\varepsilon>0$ such that $$F(x,u,v)\leq C_{0}(1+|u|^{\bar{r}}+|v|^{\bar{s}})$$
  \indent \indent  for all $|u|,|v| \leq \varepsilon$; \\
\noindent$(H4)$  There exist $R>0,0<\theta_{1}<\frac{1}{p}$ and $0<\theta_{2}<\frac{1}{q}$ such that
  $$0<F(x,u,v)\leq \theta_{1}uf(x,u,v)+\theta_{2}vg(x,u,v)$$
 \indent  \indent  for all $|u|,|v| \geq R$.\\
Then for any $p,q\geq 2$, there exists a solution $(u,v)\in \mathbb{X}$ to the following problem
\begin{equation*}\label{E4}
\left \{
\begin{array}{lcr}
 -\Delta_{p} u+h(x)|u|^{p-2}u=f(x,u,v)  &{\rm in}& V,\\
-\Delta_{q} v+h(x)|v|^{q-2}v=g(x,u,v)  &{\rm in}& V,\\
u>0 ~and ~v>0 &{\rm in}& V.
 \end{array}
\right.
\end{equation*}
\end{theorem}

Let $\mathbb{Q}=W^{m,p}(V)\times W^{n,q}(V)$ be the space with the norm
\begin{align*}
||(u,v)||_{\mathbb{Q}}=\max\Big{\{} ||u||_{W^{m,p}(V)}, ||v||_{W^{n,q}(V)}\Big{\}}.
\end{align*}
Analogous to Theorem \ref{T13}, we have the following Theorem \ref{T23} .
\begin{theorem}\label{T23}
Let $G =(V,E)$ be a finite graph and $h(x)>0$ for all $x\in V$.
 Suppose that $f(x,u): V\times \mathbb{R}\rightarrow \mathbb{R}$ and $g(x,v): V\times \mathbb{R}\rightarrow \mathbb{R}$ satisfy the following hypotheses:\\
\noindent$(H1)$ For any $x\in V$,  $f(x,t)$ and $g(x,t)$ are continuous in $t\in \mathbb{R}$,\\
\indent \indent and $f(x,0)=g(x,0)=0$ for all $x\in V$;\\
\noindent$(H2)$ $\limsup_{t\rightarrow 0}\frac{|f(x,t)|}{t}<\lambda_{mp(V)}=\inf_{u \not\equiv 0}\frac{\int_{V}(|\nabla^{m}u|^{p}+h|u|^{p})d\mu}{\int_{V}|u|^{p}d\mu}$;\\
\indent \indent$\limsup_{t\rightarrow 0}\frac{|g(x,t)|}{t}<\lambda_{nq(V)}=\inf_{v \not\equiv 0}\frac{\int_{V}(|\nabla^{n}v|^{q}+h|v|^{q})d\mu}{\int_{V}|v|^{q}d\mu}$;\\
\noindent$(H3)$ There exist $\theta_{0}>p,q$ and $s_{0}>0$ such that if $s\geq s_{0}$, then there hold\\
\indent \indent $0<\theta_{0} F(x,s)<f(x,s)s$ and $0<\theta_{0} G(x,s)<G(x,s)s$ for any $x\in V$, where \\
\indent \indent $F(x,s)=\int_{0}^{s}f(x,t)dt$ and $G(x,s)=\int_{0}^{s}g(x,t)dt$.\\
Then for $p,q\geq 2$, there exists a solution $(u,v)\in \mathbb{Q}\backslash \{(0,0)\}$ to the following problem
\begin{equation*}\label{E23}
\left \{
\begin{array}{lcr}
 \mathcal{L}_{m,p} u+h(x)|u|^{p-2}u=g(x,v)  &{\rm in}& V,\\
\mathcal{L}_{n,q} v+h(x)|v|^{q-2}v=f(x,u)  &{\rm in}& V.
 \end{array}
\right.
\end{equation*}
\end{theorem}
\section{Preliminary analysis}\label{Preliminary}
\indent In this section, we introduce some preliminary results which are crucial to prove our main theorems.

 Note that $L^{p}(V)$, $W_{0}^{m,p}(\Omega)$ and $W^{m,p}(V)$ are three Banach spaces with their norms defined as in (\ref{866.11}), (\ref{86.11}) and (\ref{866.10}) respectively.
Firstly, we introduce the following three Sobolev embedding theorems (Lemma \ref{qianru}-\ref{LEMMA1}) on graphs.

\begin{lemma} $(\cite{AGY} ~Theorem ~7)$ \label{qianru}
Let $G =(V,E)$ be a locally finite graph and $\Omega$ be a bounded domain of $V$ with $\Omega^{0}\neq \emptyset$.
Let $m$ be any positive integer and $p > 1$. Then
 $W_{0}^{m,p}(\Omega)$ is embedded in $ L^{q}(\Omega)$ for all
$1\leq  q \leq +\infty$. In particular, there exists a constant $C$ depending only on $p$, $m$ and $\Omega$
 such that for all $u\in W_{0}^{m,p}(\Omega)$
$$(\int_{\Omega}|u|^{q}d\psi)^{1/q}\leq C(\int_{\Omega}|\nabla^{m} u|^{p}d\psi)^{1/p}.$$
 Moreover, $W_{0}^{m,p}(\Omega)$ is pre-compact, namely, if $\{u_{k}\}$ is
bounded in $W_{0}^{m,p}(\Omega)$, then up to a subsequence, still denoted by $\{u_{k}\}$, there exists some $u\in W_{0}^{m,p}(\Omega)$
 such that $u_{k}\rightarrow u$ in $ W_{0}^{m,p}(\Omega)$.
\end{lemma}

\begin{lemma} $(\cite{AGY} ~Theorem ~8)$ \label{qianru2}
Let $G =(V,E)$ be a finite graph. Then for all $s>1$,  $W^{m,s}(V)$ is embedded in $ L^{q}(V)$
for all $1 \leq q \leq +\infty$, where $m$ is any positive integer. In particular, there exists a constant C depending only on $s$,$m$ and $V$
such that
$$ (\int_{V}|u|^{q}d\psi)^{1/q}\leq C(\int_{V}(|\nabla^{m} u|^{s}+h(x)|u|^{s})d\psi)^{1/s}.$$
 Moreover, $W^{m,s}(V)$ is pre-compact, namely, if $\{u_{k}\}$ is
bounded in $W^{m,s}(V)$, then us to a subsequence, there exists some $u\in W^{m,s}(V)$
 such that $u_{k}\rightarrow u$ in $W^{m,s}(V)$.
\end{lemma}

Now, we consider the Sobolev embedding theorem for $\mathcal{H}^{m,2}(V) $ defined as in (\ref{tod5}) for any integer $m\geq 1$ on graph $G=(V,E)$.
For all $ x\in V$, if we assume that there exists a constant $\omega_{0}>0$ such that $\omega_{xy}\geq\omega_{0}$ for all $y\sim x$, we can obtain that
there exists a constant $\psi_{0}>0$ such that $\psi(x)\geq\psi_{0}$ in the case of that $G = (V,E)$ is a connected finite (or locally finite) graph.

\begin{lemma}\label{LEMMA1}
Suppose that $G = (V,E)$ is a connected finite (or locally finite) graph. For all $ x\in V$, let the distance function $\rho(x) = \rho(x, O) \in L^{q}(V)$ for some $q > 0$ and $O \in V$, and assume that there exists a constant $\omega_{0}>0$ such that $\omega_{xy}\geq\omega_{0}$ for all $y\sim x$. Let $u$ be a harmonic eigenfunction of $-\Delta$ on $G$ with eigenvalue $\lambda_{1}$. Then, for any integer $m\geq 1$, we have
\begin{align*}
||u(x)||_{L^{q}(V)}\leq C^{*}||u||_{\mathcal{H}^{m,2}(V)},
\end{align*}
where
\begin{align*}
C^{*}= &\max\Big{\{}\frac{2\sqrt{2}\lambda^{\frac{1-m}{2}}_{1}}{\sqrt{\omega_{0}}}\cdot ||\rho||_{L^{q}(V)}, \frac{2^{\frac{1}{q}+1}\max\big{\{} ||\rho||_{L^{q}(V)},  \psi(O)^{\frac{1}{q}}\big{\}}}{\sqrt{ \psi(O)}}, \\
&\frac{2^{\frac{1}{q}}\max\big{\{} ||\rho||_{L^{q}(V)},  \psi(O)^{\frac{1}{q}}\big{\}}}{\sqrt{\psi_{0}(\lambda^{m}_{1}+1)}}\Big{\}}.
\end{align*}
In particular,
\begin{align}\label{888888}
||u(x)||_{L^{\infty}(V)}\leq C_{*} ||u(x)||_{\mathcal{H}^{m,2}(V)},
\end{align}
where $C_{*}=\max\Big{\{}\frac{\sqrt{2}\lambda^{\frac{1-m}{2}}_{1}}{\sqrt{\omega_{0}}}\cdot \rho(x)+\frac{1}{\sqrt{ \psi(O)}}, \frac{1}{\sqrt{\psi_{0}(\lambda^{m}_{1}+1)}}\Big{\}}$.
Moreover, $\mathcal{H}^{m,2}(V)$ is pre-compact, namely, if $\{u_{k}\}$ is
bounded in $\mathcal{H}^{m,2}(V)$, then up to a subsequence, there exists some $u\in \mathcal{H}^{m,2}(V)$
such that $u_{k}\rightarrow u$ in $\mathcal{H}^{m,2}(V)$.

\end{lemma}

\begin{proof}

Let $O$ be fixed in $V$. Choose a shortest path $P= \{x_{1}, \cdot\cdot\cdot , x_{k+1}\}$ connecting $x$ and $O$,
where $x_{1}= x, \cdot\cdot\cdot , x_{k+1} =O$, $x_{i}$ is adjacent to $x_{i+1}$ for all $1 \leq i \leq k$, and $k = \rho(x)$.

Since $\rho\in L^{q}(V)$ for some $q>0$ and $\rho(x,y)\geq1$ for all $x\neq y$, by (\ref{866.11}), we have
\begin{align} \label{10}
|||1|_{L^{q}(V)}=\big{(} \sum_{z\in V} \psi(z) \big{)}^{\frac{1}{q}}\leq& \big{(} \sum_{z\in V} \psi(z)\rho^{q} (z)+\psi(O)\big{)}^{\frac{1}{q}} \notag\\
\leq & 2^{\frac{1}{q}}\max\big{\{}\big{(} \sum_{z\in V} \psi(z)\rho^{q} (z)\big{)}^{\frac{1}{q}},  \psi(O)^{\frac{1}{q}}\big{\}}\notag\\
\leq & 2^{\frac{1}{q}}\max\big{\{}||\rho||_{L^{q}(V)},  \psi(O)^{\frac{1}{q}}\big{\}}.
\end{align}

Noting that for all $m\geq 1$,
\begin{align*}
||u||_{\mathcal{H}^{m,2}(V)}= &\Big{(}\int_{V}(|\nabla^{m}u|^{2}+|u|^{2})d\psi\Big{)}^{\frac{1}{2}} \notag\\
 \geq & \Big{(}\sum_{z\in V}|u(z)|^{2}\psi(z)\Big{)}^{\frac{1}{2}}\notag\\
  \geq &\sqrt{\psi(O)}|u(O)|,
\end{align*}
we have
\begin{align} \label{13}
|u(O)|\leq  \frac{1}{\sqrt{ \psi(O)}}||u||_{\mathcal{H}^{m,2}(V)}.
\end{align}

For any $u \in \mathcal{H}^{m,2}(V)$, it is easy to see that
\begin{align} \label{11}
|u(x)|\leq |u(x_{1})-u(x_{2})|+\cdot\cdot\cdot+|u(x_{k})-u(x_{k+1})|+|u(O)|.
\end{align}

Letting $u$ denote a harmonic eigenfunction achieving $\lambda_{1}$ ,
 for any vertex $x\in V$, by (\ref{1.13}), we have
\begin{align}
-\Delta u(x)=\lambda_{1} u(x),\cdots,
\Delta^{n} u(x)=\Delta\big{(}\Delta^{n-1} u(x)\big{)}=(-1)^{n}\lambda^{n}_{1} u(x),
\end{align}
 where $n$ is a positive integer.

In view of (\ref{86.6}), when $m$ is an odd number, we have
\begin{align}
||u||_{\mathcal{H}^{m,2}(V)}&=\big{(}\int_{V}(|\nabla^{m}u|^{2}+|u|^{2})d\psi\big{)}^{\frac{1}{2}} \notag\\
 & \geq \bigg{(}\sum_{z\in V} |\nabla^{m} u(z)|^{2}\psi(z)\bigg{)}^{\frac{1}{2}}\notag\\
 &=\bigg{(}\sum_{z\in V} |\nabla\Delta^{\frac{m-1}{2}} u(z)|^{2}\psi(z)\bigg{)}^{\frac{1}{2}}\notag\\
 &=  \Bigg{(}\sum_{z\in V} \psi(z) \cdot \Big{(}\frac{1}{2\psi(z)} \sum_{y\sim z}\omega_{yz}\big{(}\Delta^{\frac{m-1}{2}}u(y)-\Delta^{\frac{m-1}{2}}u(z)\big{)}^{2}  \Big{)}\Bigg{)}^{\frac{1}{2}}\notag\\
 &=  \Bigg{(}\sum_{z\in V} \psi(z) \cdot\lambda^{m-1}_{1}\cdot \Big{(}\frac{1}{2\psi(z)} \sum_{y\sim z}\omega_{yz}\big{(}u(y)-u(z)\big{)}^{2} \Big{)}\Bigg{)}^{\frac{1}{2}}\notag\\
&\geq  \frac{\sqrt{2}}{2} \cdot\lambda_{1}^{\frac{m-1}{2}}\max_{1\leq i\leq k}\sqrt{\omega_{x_{i}x_{i+1}}} |u(x_{i})-u(x_{i+1})|  \notag\\
&\geq \frac{\sqrt{2}}{2} \cdot\lambda^{\frac{m-1}{2}}_{1}\cdot\sqrt{\omega_{0}}\max_{1\leq i\leq k}|u(x_{i})-u(x_{i+1})|,
\end{align}
and this together with(\ref{13}) and (\ref{11}), we get
\begin{align} \label{30}
|u(x)|\leq &  k\max_{1\leq i\leq k}|u(x_{i})-u(x_{i+1})|+|u(O)| \notag\\
\leq &\big{(}\frac{\sqrt{2}\lambda^{\frac{1-m}{2}}_{1}}{\sqrt{\omega_{0}}}\cdot \rho(x)+\frac{1}{\sqrt{ \psi(O)}}\big{)}||u||_{\mathcal{H}^{m,2}(V)}.
\end{align}
 When $m$ is an even number, we have
\begin{align}\label{60}
||u||_{\mathcal{H}^{m,2}(V)}= &\Big{(}\int_{V}(|\nabla^{m}u|^{2}+|u|^{2})d\psi\Big{)}^{\frac{1}{2}} \notag\\
 =& \bigg{(}\sum_{z\in V}\big{(} |\Delta^{\frac{m}{2}} u(z)|^{2}+|u(z)|^{2}\big{)}\psi(z)\bigg{)}^{\frac{1}{2}}\notag\\
  =& \bigg{(}\sum_{z\in V}\big{(} \lambda^{m}_{1}|u(z)|^{2}+|u(z)|^{2}\big{)}\psi(z)\bigg{)}^{\frac{1}{2}}\notag\\
 \geq&\sqrt{\psi_{0}(\lambda^{m}_{1}+1)}|u(x)|,
\end{align}
and thus, by (\ref{60}), we get
\begin{align}\label{61}
|u(x)|\leq \frac{1}{\sqrt{\psi_{0}(\lambda^{m}_{1}+1)}}||u||_{\mathcal{H}^{m,2}(V)}.
\end{align}
By (\ref{30}) and (\ref{61}), for anr integer $m\geq 1$, we have
\begin{align*}
|u(x)|\leq& \max
\Big{\{}\frac{\sqrt{2}\lambda^{\frac{1-m}{2}}_{1}}{\sqrt{\omega_{0}}}\cdot \rho(x)+\frac{1}{\sqrt{ \psi(O)}}, \frac{1}{\sqrt{\psi_{0}(\lambda^{m}_{1}+1)}}\Big{\}}||u||_{\mathcal{H}^{m,2}(V)}\\
\leq& \max
\Big{\{}\frac{2\sqrt{2}\lambda^{\frac{1-m}{2}}_{1}}{\sqrt{\omega_{0}}}\cdot \rho(x), \frac{2}{\sqrt{ \psi(O)}}, \frac{1}{\sqrt{\psi_{0}(\lambda^{m}_{1}+1)}}\Big{\}}||u||_{\mathcal{H}^{m,2}(V)},
\end{align*}
which implies $(\ref{888888})$.
Thus, we have
\begin{align*}
||u(x)||_{L^{q}(V)}=&(\int_{V}|u(x)|^{q}d\psi)^{\frac{1}{q}}\notag\\
\leq& \max
\Big{\{}\frac{2\sqrt{2}\lambda^{\frac{1-m}{2}}_{1}}{\sqrt{\omega_{0}}}\cdot ||\rho||_{L^{q}(V)}, \frac{2||1||_{L^{q}(V)}}{\sqrt{ \psi(O)}}, \frac{||1||_{L^{q}(V)}}{\sqrt{\psi_{0}(\lambda^{m}_{1}+1)}}\Big{\}}||u||_{\mathcal{H}^{m,2}(V)}\\
\leq& \max
\Big{\{}\frac{2\sqrt{2}\lambda^{\frac{1-m}{2}}_{1}}{\sqrt{\omega_{0}}}\cdot ||\rho||_{L^{q}(V)}, \frac{2^{\frac{1}{q}+1}\max\big{\{} ||\rho||_{L^{q}(V)},  \psi(O)^{\frac{1}{q}}\big{\}}}{\sqrt{ \psi(O)}}, \\
&\frac{2^{\frac{1}{q}}\max\big{\{} ||\rho||_{L^{q}(V)},  \psi(O)^{\frac{1}{q}}\big{\}}}{\sqrt{\psi_{0}(\lambda^{m}_{1}+1)}}\Big{\}}||u||_{\mathcal{H}^{m,2}(V)}\\
&= C^{*}||u||_{\mathcal{H}^{m,2}(V)},\notag
\end{align*}
where
\begin{align*}
C^{*}= &\max\Big{\{}\frac{2\sqrt{2}\lambda^{\frac{1-m}{2}}_{1}}{\sqrt{\omega_{0}}}\cdot ||\rho||_{L^{q}(V)}, \frac{2^{\frac{1}{q}+1}\max\big{\{} ||\rho||_{L^{q}(V)},  \psi(O)^{\frac{1}{q}}\big{\}}}{\sqrt{ \psi(O)}}, \\
&\frac{2^{\frac{1}{q}}\max\big{\{} ||\rho||_{L^{q}(V)},  \psi(O)^{\frac{1}{q}}\big{\}}}{\sqrt{\psi_{0}(\lambda^{m}_{1}+1)}}\Big{\}}.
\end{align*}
\end{proof}

\begin{remark}\label{remark1}
Let $G = (V,E)$ be a connected finite (or locally finite) graph and $\Omega$ be a bounded domain of $V$ with $\Omega^{\circ}\neq\emptyset$.
By Lemma \ref{qianru} and Lemma $\ref{LEMMA1}$, since
$$(\int_{\Omega}|u|^{q}d\psi)^{1/q}\leq C(\int_{\Omega}|\nabla^{m} u|^{p}d\psi)^{1/p},$$
the norm
\begin{align*}
||u||_{\mathcal{H}_{0}^{m,2}(\Omega)}=\Big{(}\int_{\Omega}(|\nabla^{m}u|^{2}+|u|^{2})d\psi\Big{)}^{\frac{1}{2}}
\end{align*}
is equivalent to the norm
\begin{align*}
||u||_{\mathcal{\mathcal{H}}_{0}^{m,2}(\Omega)}=\Big{(}\int_{\Omega}|\nabla^{m}u|^{2}d\psi\Big{)}^{\frac{1}{2}},
\end{align*}
where $p>1$, $q\geq1$ and integer $m\geq 1$.
\end{remark}

\begin{remark}\label{remark2}
Let $G = (V,E)$ be a connected finite (or locally finite) graph and $\Omega$ be a bounded domain of $V$ with $\Omega^{\circ}\neq\emptyset$.
From Lemma $\ref{qianru}$-$\ref{LEMMA1}$, we have the following results:

$(1)$  Set $\maltese=W^{m,p}(V)\times W^{n,q}(V)$ and the norm on $\maltese$ is defined as
\begin{align*}
||(u,v)||_{\maltese}=\max\big{\{} ||u||_{W^{m,p}(V)}, ||v||_{ W^{n,q}(V)}\big{\}}.
\end{align*}
Take a sequence of functions
$(\hat{u}_{j}, \hat{v}_{j})\in \maltese$.
 If there exist $ u\in W^{m,p}(V)$ and $v\in W^{n,q}(V)$ such that $||\hat{u}_{j}-u||_{W^{m,p}(V)}\rightarrow 0$ and
 $||\hat{v}_{j}-v||_{W^{n,q}(V)}\rightarrow 0$ as $j\rightarrow \infty$,
we can easily get that
$||(\hat{u}_{j},\hat{v}_{j})-(u,v)||_{\maltese}= \max\big{\{}||\hat{u}_{j}-u||_{W^{m,p}(V)}, ||\hat{v}_{j}-v||_{W^{n,q}(V)}\big{\}}\rightarrow 0$ as $j\rightarrow \infty$,
 Thus, $(\hat{u}_{j},\hat{v}_{j})\rightarrow(u,v)$ in $\maltese$ as $j\rightarrow \infty$. In addition, if the norm on $\maltese$ is defined as
\begin{align*}
||(u,v)||_{\maltese}= ||u||_{W^{m,p}(V)}+ ||v||_{ W^{n,q}(V)},
\end{align*}
we can easily get the same result.

$(2)$ The results in $(1)$ are also applicable to the spaces $\maltese_{1}=W_{0}^{m,p}(\Omega)\times W_{0}^{n,q}(\Omega)$,
     $\maltese_{2}=\mathcal{H}^{m,2}(V)\times \mathcal{H}^{n,2}(V)$
    and $\maltese_{3}=\mathcal{H}_{0}^{m,2}(\Omega)\times \mathcal{H}_{0}^{n,2}(\Omega)$.

\end{remark}

Now, we get a version of the strong maximum principle on graphs (Lemma \ref{zhunze1}).
\begin{lemma}$(Strong ~maximum~ principle)$ \label{zhunze1}
Let $G = (V, E)$ be a connected and finite (or locally finite) graph.
Assume that $u(x)\geq 0$, $v(x)\geq 0$, $p,q\geq 2$ and
\begin{equation*}
\left \{
\begin{array}{lcr}
 -\Delta_{p} u + h_{1}(x)|u|^{p-2}u\geq 0, \\
-\Delta_{q} v + h_{2}(x)|v|^{q-2}v\geq 0
 \end{array}
\right.
\end{equation*}
for all $x\in V$,
where $h_{1}(x)$ and $h_{2}(x)$ are any two functions defined on $V$.
If there exist $x_{0}\in V$ such that $u(x_{0}) = 0$ and $x_{*}\in V$ such that $v(x_{*}) = 0$ , then $\big{(}u(x),v(x)\big{)}\equiv (0,0)$ for all $x\in V$.
\end{lemma}

\begin{proof}
By (\ref{Lap}), since $u(x_{0})=0$, we have
\begin{align*}
-\Delta_{p}& u(x_{0})+ h_{1}(x_{0})|u(x_{0})|^{p-2}u(x_{0})\\ \notag
=&\frac{1}{2\psi(x_{0})}\sum_{y\sim x_{0}}\omega_{x_{0}y}\Big{(}|\nabla u|^{p-2}(y)+|\nabla u|^{p-2}(x_{0})\Big{)}\Big{(}u(x_{0})-u(y)\Big{)}\geq 0.
\end{align*}
Thus, we have
\begin{align*}
\frac{1}{2\psi(x_{0})}\sum_{y\sim x_{0}}\omega_{x_{0}y}\Big{(}|\nabla u|^{p-2}(y)+|\nabla u|^{p-2}(x_{0})\Big{)}u(y)\leq 0.
\end{align*}
Since $u(y)\geq 0$ for all $y\in V$ and $\omega_{yx_{0}} > 0$ for all $y \sim x_{0}$, we get
\begin{align*}
u(y)\equiv 0, ~for~ all~ y \sim x_{0}.
\end{align*}
Since $G$ is connected, repeat the above process and we have
\begin{align*}
u(y)\equiv 0, ~for~ all~ y\in V.
\end{align*}
Similarly, by (\ref{Lap}), we have
\begin{align*}
-\Delta_{q}& v(x_{*})+ h_{2}(x_{*})|v(x_{*})|^{q-2}v(x_{*})\\ \notag
=&\frac{1}{2\psi(x_{*})}\sum_{y\sim x_{*}}\omega_{x_{*}y}\Big{(}|\nabla v|^{q-2}(y)+|\nabla v|^{q-2}(x_{*})\Big{)}\Big{(}v(x_{*})-v(y)\Big{)}\geq 0.
\end{align*}
Thus, we have
\begin{align*}
\frac{1}{2\psi(x_{*})}\sum_{y\sim x_{*}}\omega_{x_{*}y}\Big{(}|\nabla v|^{q-2}(y)+|\nabla v|^{q-2}(x_{*})\Big{)}v(y)\leq 0.
\end{align*}
Since $v(y)\geq 0$ and $\omega_{x_{*}y} > 0$ for all $y \sim x_{*}$, we get
\begin{align*}
v(y)\equiv 0, ~for~ all~ y \sim x_{*}.
\end{align*}
Since $G$ is connected, repeat the above process and we have
\begin{align*}
v(y)\equiv 0, ~for~ all~ y\in V.
\end{align*}
Therefore, $\big{(}u(x),v(x)\big{)}\equiv (0,0)$ for all $x\in V$.
\end{proof}

Now, we introduce the following Palais-Smale condition (Definition \ref{PS}) and the Mountain pass lemma (Lemma \ref{mountain}) for elliptic systems on graphs.

\begin{definition} (\cite{YLL})  \label{PS}
Let $(X, ||\cdot ||)$ be a Banach space and $J(u,v) \in C^{1}(X,\mathbb{R})$. The functional $J(u,v)$ is said to satisfy the $(PS)_{c}$ condition if for any
sequence $\{(u_{k},v_{k})\}\subset X$ that satisfies $J(u_{k},v_{k})\leq c$ and $J'(u_{k},v_{k})\rightarrow 0$ as $k\rightarrow \infty$, $\{(u_{k},v_{k})\}$ has a strongly convergent subsequence.
\end{definition}

\begin{lemma}\label{mountain}
(Ambrosetti-Rabinowitz ~\cite{YLL}).
Let $(X, ||\cdot ||)$ be a Banach space and $J(u,v) \in C^{1}(X,\mathbb{R})$. There exist $(u_{0},v_{0}) \in X$ and $r > 0$
such that when $||(u_{0},v_{0})|| > r$ there holds
$$ b=\inf_{||(u,v)||=r}J(u,v)>J(0,0)>J(u_{0},v_{0}).$$
Moreover, if $J$ satisfies the $(PS)_{c}$ condition with
 $c=\inf_{\gamma\in\Gamma}\max_{t\in[0,1]}J(\gamma(t))$,
where $$\Gamma=\{\gamma\in C([0,1], X): \gamma(0)=(0,0),\gamma(1)=(u_{0},v_{0})\},$$
then c is a critical value of J.
\end{lemma}

\section{Proof of the global existence results for the Toda system on graphs}
In this section, by using a direct variational method, Lemma \ref{LEMMA1}, Remark \ref{remark1} and Remark \ref{remark2},
we prove the global existence results for the Toda system on graphs (Theorem \ref{EEET23}-\ref{EEET238}).

\noindent $\mathbf{Proof~ of~ Theorem ~\ref{EEET23}.}$
Now, we define the functional $J_{1}:\mathcal{H}\rightarrow \mathbb{R}$ by
\begin{align} \label{NNN}
J_{1}(u,v)=&\frac{1}{2}\int_{V}|\nabla^{m} u|^{2}d\psi
+\frac{1}{2}\int_{V}|\nabla^{n} v|^{2}d\psi \notag \\
&-\frac{1}{2}\varphi_{1}log \int_{V}e^{2u-v}d\psi-\frac{1}{2}\varphi_{2}log \int_{V}e^{-u+2v}d\psi.
\end{align}
 Obviously, $J_{1}\in C^{1}\big{(}\mathcal{H},\mathbb{R}\big{)}$.

 By the Cauchy's inequality, (\ref{866.121212}), Remark \ref{remark1} and Lemma \ref{LEMMA1}, for any $\varepsilon>0$, we have
 \begin{align*}
u\leq& \varepsilon\int_{V}|\nabla^{m}u|^{2}d\psi+\frac{1}{4\varepsilon}\Big{(} \frac{u}{(\int_{V}|\nabla^{m}u|^{2}d\psi)^{\frac{1}{2}}}\Big{)}^{2}\\
\leq& \varepsilon\int_{V}|\nabla^{m}u|^{2}d\psi+C_{0},
\end{align*}
 and
\begin{align*}
-u\leq& \varepsilon\int_{V}|\nabla^{m}u|^{2}d\psi+\frac{1}{4\varepsilon}\Big{(} \frac{-u}{(\int_{V}|\nabla^{m}u|^{2}d\psi)^{\frac{1}{2}}}\Big{)}^{2}\\
\leq& \varepsilon\int_{V}|\nabla^{m}u|^{2}d\psi+C_{0},
\end{align*}
 where $C_{0}$ is a constant. Thus, we get
\begin{align*}
log \int_{V}e^{2u-v}d\psi\leq 3C_{0}+2\varepsilon||u||_{W^{m,2}(V)}^{2}+\varepsilon||v||_{W^{n,2}(V)}^{2},\\
log \int_{V}e^{-u+2v}d\psi\leq 3C_{0}+\varepsilon||u||_{W^{m,2}(V)}^{2}+2\varepsilon||v||_{W^{n,2}(V)}^{2}.
\end{align*}
Therefore, we have
 \begin{align}\label{to1}
J_{1}(u,v)\geq &\Big{(} \frac{1}{2}-(\varphi_{1}+\frac{1}{2}\varphi_{2})\varepsilon \Big{)}||u||_{W^{m,2}(V)}^{2}+\Big{(} \frac{1}{2}-(\frac{1}{2}\varphi_{1}+\varphi_{2})\varepsilon \Big{)}||v||_{W^{n,2}(V)}^{2}\notag\\
&-\frac{3}{2}C_{0}(\varphi_{1}+\varphi_{2})
\end{align}
Choosing $\varepsilon$ such that $0<\varepsilon<\min\{\frac{1}{2\varphi_{1}+\varphi_{2}},\frac{1}{\varphi_{1}+2\varphi_{2}}\}$, we get
\begin{align*}
J_{1}(u,v)\geq -\frac{3}{2}C_{0}(\varphi_{1}+\varphi_{2}).
\end{align*}
Therefore, $J_{1}(u,v)$ is bounded from below in $\mathcal{H}$. So there exists a sequence $\{(u_{n},v_{n})\}\subset \mathcal{H}$ such that
\begin{align}\label{to2}
\lim_{n\rightarrow \infty} J_{1}(u_{n},v_{n})=\inf_{\mathcal{H}}J_{1}(u,v).
\end{align}
On the other hand,
\begin{align}\label{to3}
 J_{1}(0,0)=&-\frac{1}{2}\varphi_{1}log \int_{V}1d\psi-\frac{1}{2}\varphi_{2}log \int_{V}1d\psi \notag\\
\leq &-\frac{1}{2}\varphi_{1}log \psi_{0}-\frac{1}{2}\varphi_{2}log \psi_{0}.
\end{align}
By (\ref{to1}),(\ref{to2}) and (\ref{to3}), when $n$ is sufficiently large, there exists  $\varepsilon>0$ such that
\begin{align}\label{to5}
 -\frac{3}{2}C_{0}(\varphi_{1}+\varphi_{2})\leq J_{1}(u_{n},v_{n})&\leq \inf_{\mathcal{H}}J_{1}(u,v)+\varepsilon \leq J_{1}(0,0)+\varepsilon\notag\\
 &\leq -\frac{1}{2}\varphi_{1}log \psi_{0}-\frac{1}{2}\varphi_{2}log \psi_{0}+\varepsilon.
\end{align}
By (\ref{to1}), (\ref{to5}) and Lemma \ref{LEMMA1}, we get that $\{(u_{n},v_{n})\}$ is bounded in $\mathcal{H}$.
Therefore, up to a subsequence, still denoted by $\{(u_{n},v_{n})\}$, there exists $(u_{0},v_{0})$ such that
\begin{align}\label{to6}
(u_{n},v_{n})\rightarrow (u_{0},v_{0})~~as~~n\rightarrow\infty.
\end{align}
Setting
\begin{align*}
-\Delta u_{n}=\lambda_{1}u_{n}~~ and -\Delta v_{n}=\lambda_{1}v_{n}
\end{align*}
and letting $n\rightarrow \infty$, as $(u_{n},v_{n})\rightarrow (u_{0},v_{0})$, we get
\begin{align*}
-\Delta u_{0}=\lambda_{1}u_{0}~~ and -\Delta v_{0}=\lambda_{1}v_{0}.
\end{align*}
Thus, we get
\begin{align} \label{to338}
(u_{0},v_{0})\in \mathcal{H}
\end{align}

Calculating directly, we have
\begin{align}\label{to35}
&\Big{|}\int_{V}(e^{2u_{n}-v_{n}}-e^{2u_{0}-v_{0}})d\psi\Big{|} \notag \\
=&\Big{|}\int_{V}d\psi\int^{1}_{0}\frac{d}{dt}e^{t\big{(}(2u_{n}-v_{n})-(2u_{0}-v_{0})\big{)}+(2u_{0}-v_{0})}dt\Big{|}\notag \\
=&\Big{|}\int^{1}_{0}dt\int_{V}e^{t\big{(}(2u_{n}-v_{n})-(2u_{0}-v_{0})\big{)}+(2u_{0}-v_{0})}
\cdot\big{(}(2u_{n}-v_{n})-(2u_{0}-v_{0})\big{)}d\psi\Big{|}\notag \\
\leq&\Big{|}\int^{1}_{0}\Big{(}\int_{V}e^{2t\big{(}(2u_{n}-v_{n})-2(2u_{0}-v_{0})\big{)}+(2u_{0}-v_{0})}d\psi\Big{)}^{\frac{1}{2}}
\cdot\Big{(}\int_{V}\big{(}(2u_{n}-v_{n})\notag \\
&-(2u_{0}-v_{0})\big{)}^{2}d\psi\Big{)}^{\frac{1}{2}}dt\Big{|}\notag \\
\leq& C \Big{(}\int_{V}\big{(}(2u_{n}-v_{n})-(2u_{0}-v_{0})\big{)}^{2}d\psi\Big{)}^{\frac{1}{2}}\rightarrow 0,~~as~~n\rightarrow\infty,
\end{align}
where C is a positive constant. Similarly, we have
\begin{align}\label{to36}
\Big{(}\int_{V}\big{(}(-u_{n}+2v_{n})-(-u_{0}+2v_{0})\big{)}^{2}d\psi\Big{)}^{\frac{1}{2}}\rightarrow 0,~~as~~n\rightarrow\infty.
\end{align}
 Therefore, by (\ref{to6}),(\ref{to338}), (\ref{to35}) and (\ref{to36}), one obtains
\begin{align*}
J_{1}(u_{0},v_{0})=\lim_{n\rightarrow \infty} J_{1}(u_{n},v_{n})=\inf_{\mathcal{H}}J_{1}(u,v).
\end{align*}
By a direct variational method, we get that $(u_{0},v_{0})\in \mathcal{H}$ satisfies the equation ($\ref{EEERRR}$).
Thus, the proof is completed.  ~ ~$\Box$ \\

\noindent $\mathbf{Proof~ of~ Theorem ~\ref{EEET238}.}$
Set $\mathcal{H}_{k}=\mathcal{H}_{0}^{m,2}(V_{k})\times \mathcal{H}_{0}^{n,2}(V_{k})$ with the norm
\begin{align*}
||(u,v)||_{\mathcal{H}_{k}}=\max\big{\{} ||u||_{\mathcal{H}^{m,2}_{0}(V_{k})}, ||v||_{\mathcal{H}^{n,2}_{0}(V_{k})}\big{\}}.
\end{align*}
Define the functional $J_{k(2)}:\mathcal{H}_{k}\rightarrow \mathbb{R}$ by
\begin{align} \label{to21}
J_{k(2)}(u,v)=&\frac{1}{2}\int_{V_{k}}|\nabla^{m} u|^{2}d\psi
+\frac{1}{2}\int_{V_{k}}|\nabla^{n} v|^{2}d\psi \notag \\
&-\frac{1}{2}\varphi_{1}log \int_{V_{k}}e^{2u-v}d\psi-\frac{1}{2}\varphi_{2}log \int_{V_{k}}e^{-u+2v}d\psi.
\end{align}
 Obviously, $J_{k(2)}\in C^{1}\big{(}\mathcal{H}_{k},\mathbb{R}\big{)}$.

By Remark \ref{remark1}, the norm
\begin{align*}
||u||_{\mathcal{H}^{m,2}_{0}(V_{k})}=\Big{(}\int_{V_{k}}(|\nabla^{m}u|^{2}+|u|^{2})d\psi\Big{)}^{\frac{1}{2}}
\end{align*}
is equivalent to the norm
\begin{align*}
||u||_{\mathcal{H}^{m,2}_{0}(V_{k})}=\Big{(}\int_{V_{k}}|\nabla^{m}u|^{2}d\psi\Big{)}^{\frac{1}{2}}.
\end{align*}
Therefore, by Lemma \ref{LEMMA1}, we get
\begin{align*}
||u(x)||_{L^{q}(V_{k})}\leq C^{*}||u||_{\mathcal{H}_{0}^{m,2}(V_{k})}, ~\forall u\in \mathcal{H}_{0}^{m,2}(V_{k}).
\end{align*}
Just as in the proof of Theorem \ref{EEET23},
we can get that there exist a sequence $\{(\breve{u}_{n},\breve{v}_{n})\}\subset \mathcal{H}_{k}$ and $\varepsilon>0$ such that
\begin{align}\label{to23}
 -\frac{3}{2}C_{0}(\varphi_{1}+\varphi_{2})\leq J_{k(2)}(\breve{u}_{n},\breve{v}_{n})&\leq \inf_{\mathcal{H}_{k}}J_{k(2)}(u,v)+\varepsilon \leq J_{k(2)}(0,0)+\varepsilon\notag\\
 &\leq -\frac{1}{2}\varphi_{1}log \psi_{0}-\frac{1}{2}\varphi_{2}log \psi_{0}+\varepsilon
\end{align}
when $n$ is sufficiently large.
Therefore, we get that $\{(\breve{u}_{n},\breve{v}_{n})\}$ is bounded in $\mathcal{H}_{k}$.
Thus, up to a subsequence, still denoted by $\{(\breve{u}_{n},\breve{v}_{n})\}$, there exists $(u_{k},v_{k})\in \mathcal{H}_{k}$ such that
\begin{align}\label{to26}
(\breve{u}_{n},\breve{v}_{n})\rightarrow (u_{k},v_{k})~~ in~~\mathcal{H}_{k}~~as~~n\rightarrow\infty.
\end{align}
 Therefore, one obtains
\begin{align}\label{to212}
J_{k(2)}(u_{k},v_{k})=\lim_{n\rightarrow \infty} J_{k(2)}(\breve{u}_{n},\breve{v}_{n})=\inf_{\mathcal{H}_{k}}J_{k(2)}(u,v).
\end{align}
By a direct variational method, we get that $(u_{k},v_{k})$ satisfies the Euler-Lagrange equation
\begin{equation}\label{toEEERRR}
\left \{
\begin{array}{lcr}
 \mathcal{L}_{m,2} u_{k}=\varphi_{1}(\frac{e^{2u_{k}-v_{k}}}{\int_{V}e^{2u_{k}-v_{k}}d\psi}-1) ~ &{\rm in}& V_{k},\\
\mathcal{L}_{n,2} v_{k}=\varphi_{2}(\frac{e^{-u_{k}+2v_{k}}}{\int_{V}e^{-u_{k}+2v_{k}}d\psi}-1)~  &{\rm in}& V_{k},\\
u_{k}=v_{k}=0 &{\rm on}& \partial V_{k}.
 \end{array}
\right.
\end{equation}

Let $\Theta$ denote any finite set of $V$. By (\ref{to23}) and (\ref{to212}), we get
\begin{align}\label{to213}
||u_{k}||^{2}_{W_{0}^{m,2}(V_{k})} \leq M_{1} ~~and~~ ||v_{k}||^{2}_{W_{0}^{n,2}(V_{k})}\leq M_{1},
\end{align}
where $M_{1}$ is a positive constant independent of $k$.
When $k$ is large enough, we can get that $\Theta\subset V_{k}$. By (\ref{LEMMA1}), (\ref{to213}), Lemma \ref{LEMMA1} and Remark \ref{remark2}, we have
\begin{align*}
||u_{k}||_{L^{\infty}(\Theta)} \leq M_{2} ~~and ~~ ||v_{k}||_{L^{\infty}(\Theta)}\leq M_{2},
\end{align*}
where $M_{2}$ is a positive constant independent of $k$.
Thus, $\{(u_{k},v_{k} )\}$ is uniformly bounded in $\Theta$.
Therefore, we have that $\{(u_{k},v_{k} )\}$ is uniformly bounded in $V_{1}$. Thus, there exists a subsequence of $\{(u_{k},v_{k} )\}$,
denoted by $\{(u_{1k},v_{1k} )\}$, and functions $(u^{*}_{1},v^{*}_{1} )$ such that $(u_{1k},v_{1k} )\rightarrow (u^{*}_{1},v^{*}_{1} )$ in $V_{1}$.
 Again, $\{(u_{1k},v_{1k} )\}$ is uniformly bounded in $V_{2}$. Then there exists a subsequence of $\{(u_{1k},v_{1k} )\}$,
denoted by $\{(u_{2k},v_{2k} )\}$, and functions $(u^{*}_{2},v^{*}_{2} )$ such that $(u_{2k},v_{2k} )\rightarrow (u^{*}_{2},v^{*}_{2} )$ in $V_{2}$.
 Obviously, $(u^{*}_{1},v^{*}_{1} )=(u^{*}_{2},v^{*}_{2})$ in $V_{1}$. Repeating this process, we can find a diagonal subsequence $\{(u_{kk},v_{kk} )\}$, which is still denoted by $\{(u_{k},v_{k} )\}$, and functions $(u^{*},v^{*}):V\times V\rightarrow \mathbb{R}$ such that for any finite set $\Theta\subset V$,
$(u_{k},v_{k})\rightarrow (u^{*},v^{*})$ in $\Theta$. For any fixed $x \in V$, calculating similarly as (\ref{to35}) and letting $k \rightarrow \infty$ in (\ref{toEEERRR}), we obtain
\begin{equation*}
\left \{
\begin{array}{lcr}
 \mathcal{L}_{m,2} u^{*}=\varphi_{1}(\frac{e^{2u^{*}-v^{*}}}{\int_{V}e^{2u^{*}-v^{*}}d\psi}-1) ~ &{\rm in}& V,\\
\mathcal{L}_{n,2} v^{*}=\varphi_{2}(\frac{e^{-u^{*}+2v^{*}}}{\int_{V}e^{-u^{*}+2v^{*}}d\psi}-1)~  &{\rm in}& V.\\
 \end{array}
\right.
\end{equation*}
Noting that when $k\rightarrow \infty$ there hold
$-\Delta u^{*}=\lambda_{1}u^{*}$ and $-\Delta v^{*}=\lambda_{1}v^{*}$ as $(u_{k},v_{k})\rightarrow (u^{*},v^{*})$,
we get that $(u^{*},v^{*})\in \mathcal{H}$.
Therefore, $(u^{*},v^{*})\in \mathcal{H}$ is a solution to the problem (\ref{EEERRR}).
Thus, the proof is completed.  ~ ~$\Box$

\section{Proof of the Local existence results on locally finite graphs}

In this section, by using the mountain pass theorem and some Lemmas in section \ref{Preliminary}, we prove the Local existence results (Theorem \ref{T1}-\ref{T13}) on locally finite graphs.

\noindent $\mathbf{Proof~ of~ Theorem ~\ref{T1}.}$
Now, we define the functional $J_{3}:\mathbb{H}\rightarrow \mathbb{R}$ by
\begin{align*}
J_{3}(u,v)=\frac{1}{2}||(u,v)||^{2}_{\mathbb{H}}-\frac{1}{p+q}\int_{\Omega}|u|^{p}|v|^{q}d\psi.
\end{align*}
Firstly, we prove that $J_{3}(u,v)$ satisfies the $(PS)_{c}$ condition for any $c\in \mathbb{R}$.
To see this, we take a sequence $\{(u_{k},v_{k})\}$ in $\mathbb{H}$ such that $J_{3}(u_{k},v_{k})\rightarrow c$ and $J_{3}'(u_{k},v_{k})\rightarrow 0$
as $k\rightarrow +\infty$. Thus, we have
\begin{align}
\frac{1}{2}||(u_{k},v_{k})||^{2}_{\mathbb{H}}-\frac{1}{p+q}\int_{\Omega}|u_{k}|^{p}|v_{k}|^{q}d\psi\leq c+o_{k}(1)||(u_{k},v_{k})||_{\mathbb{H}},\label{3.11}\\
||(u_{k},v_{k})||^{2}_{\mathbb{H}}-\int_{\Omega}|u_{k}|^{p}|v_{k}|^{q}d\psi =o_{k}(1)||(u_{k},v_{k})||_{\mathbb{H}}.\label{3.12}
\end{align}
By (\ref{3.11}) and (\ref{3.12}), we have
\begin{align} \label{3.13}
(\frac{1}{2}-\frac{1}{p+q})||(u_{k},v_{k})||^{2}_{\mathbb{H}}\leq M,
\end{align}
where $M$ is a positive constant.
Since $p+q>2$, by (\ref{3.13}) we get that $\{(u_{k},v_{k})\}$ is bounded in $\mathbb{H}$.
By Lemma \ref{qianru} and Remark \ref{remark2}, since $W_{0}^{1,2}(\Omega)$ is pre-compact and $\{(u_{k},v_{k})\}$ is a bounded sequence, we get that $(u_{k},v_{k})\rightarrow (u,v)$ strongly in $\mathbb{H}$. Therefore, $J_{3}(u,v)$ satisfies the $(PS)_{c}$ condition for any $c\in \mathbb{R}$.

Now, we prove that $J_{3}(u,v)$ satisfies all the conditions in the mountain pass theorem in Lemma \ref{mountain}.
Obviously, we have
\begin{align} \label{001}
J_{3}(0,0)=0.
\end{align}
From \cite{Huei}, we can set
\begin{align} \label{canshu1}
S_{1}=\inf_{(u,v)\in \mathbb{H}\backslash \{(0,0)\}}\frac{||(u,v)||^{2}_{\mathbb{H}}}{(\int_{\Omega}|u|^{p}|v|^{q}d\psi)^{\frac{2}{p+q}}}.
\end{align}
By (\ref{canshu1}), for all $(u,v)\in \mathbb{H}$, we get
\begin{align} \label{3.14}
J_{3}(u,v)=&\frac{1}{2}||(u,v)||^{2}_{\mathbb{H}}-\frac{1}{p+q}\int_{\Omega}|u|^{p}|v|^{q}d\psi\notag\\
\geq &\frac{1}{2}||(u,v)||^{2}_{\mathbb{H}}-\frac{1}{p+q}S^{-\frac{p+q}{2}}_{1}||(u,v)||^{p+q}_{\mathbb{H}}.
\end{align}
Since $p+q>2$, by (\ref{3.14}) we can find some sufficiently small $r>0$ such that
\begin{align} \label{3.15}
 \inf_{||(u,v)||_{\mathbb{H}}=r} J_{3}(u,v)>0.
\end{align}
On the other hand, for any $(u,v)\in \mathbb{H}\backslash \{(0,0)\}$, we have
\begin{align} \label{3.16}
J_{3}(tu,tv)=\frac{t^{2}}{2}||(u,v)||^{2}_{\mathbb{H}}-\frac{|t|^{p+q}}{p+q}\int_{\Omega}|u|^{p}|v|^{q}d\psi .
\end{align}
Passing to the limit $t\rightarrow +\infty$,  by (\ref{3.16}) we get
\begin{align} \label{3.17}
J_{3}(tu,tv) \rightarrow -\infty.
\end{align}
Hence there exists some $(u_{0},v_{0})\in \mathbb{H}\backslash \{(0,0)\}$ such that
\begin{align} \label{3.18}
J_{3}(u_{0},v_{0}) <0, ~~when~||(u_{0},v_{0})||_{\mathbb{H}}>r.
\end{align}
By (\ref{001}), (\ref{3.15}) and (\ref{3.18}), we conclude by Lemma \ref{mountain} that
$$c=\inf_{\gamma\in\Gamma}\max_{t\in[0,1]}J_{3}(\gamma(t))$$ is a critical value of $J_{3}$, where
$$\Gamma=\{\gamma\in C([0,1], \mathbb{H}): \gamma(0)=(0,0),\gamma(1)=(u_{0},v_{0})\}.$$
Therefore, there exists some
\begin{align} \label{co001}
(u,v)\in \mathbb{H}\backslash \{(0,0)\}
\end{align}
satisfying
\begin{equation}\label{AE1}
\left \{
\begin{array}{lcr}
 -\Delta u=\frac{p}{p+q}|u|^{p-2}u|v|^{q} &{\rm in}&\Omega^{\circ}, \\
-\Delta v=\frac{q}{p+q}|u|^{p}|v|^{q-2}v &{\rm in}&\Omega^{\circ},\\
u=v=0 &{\rm on}&\partial\Omega.
 \end{array}
\right.
\end{equation}
Since $J_{3}(u,v)=J_{3}(|u|,|v|)$, we may assume that $u\geq 0$ and $v\geq 0$. Thus, by (\ref{AE1}), we get $-\Delta u\geq 0 $ and $-\Delta v\geq 0 $.
Suppose that there exist a point $x_{0}\in \Omega^{\circ}$ such that $u(x_{0})=0=\min_{x\in \Omega^{\circ}}u(x)$, and a point $x_{*}\in \Omega^{\circ}$ such that $v(x_{*})=0=\min_{x\in \Omega^{\circ}}v(x)$.
Then, by Lemma \ref{zhunze1}, we get that $(u,v)\equiv(0,0)$, which is a contradiction with (\ref{co001}).
Therefore, $u> 0, ~v> 0$ for all $x\in \Omega^{\circ}$, and the proof is completed.  ~ ~$\Box$ \\

\noindent $\mathbf{Proof~ of~ Theorem ~\ref{T2}.}$
We define the functional $J_{4}:\mathbb{W}\rightarrow \mathbb{R}$ by
\begin{align*}
J_{4}(u,v)=\frac{\alpha+1}{p}\int_{\Omega}|\nabla u|^{p}d\psi+\frac{\beta+1}{q}\int_{\Omega}|\nabla v|^{q}d\psi - \lambda_{0}\int_{\Omega}|u|^{\alpha+1}|v|^{\beta+1}d\psi.
\end{align*}
Firstly, we prove that $J_{4}(u,v)$ satisfies the $(PS)_{c}$ condition for any $c\in \mathbb{R}$.
To see this, we take a sequence $\{(u_{k},v_{k})\}$ in $\mathbb{W}$ such that as $k\rightarrow +\infty$ we have
\begin{align}
J_{4}(u_{k},v_{k})\rightarrow c,\label{y.2}\\
J_{4}'(u_{k},v_{k})\rightarrow 0.\label{y.3}
\end{align}
 By (\ref{y.2}), we have
\begin{align} \label{y.4}
\frac{\alpha+1}{p}\int_{\Omega}|\nabla u_{k}|^{p}d\psi+\frac{\beta+1}{q}\int_{\Omega}|\nabla v_{k}|^{q}d\psi - \lambda_{0}\int_{\Omega}|u_{k}|^{\alpha+1}|v_{k}|^{\beta+1}d\psi= c+o_{k}(1).
\end{align}
By (\ref{y.3}), we have
\begin{align}
\int_{\Omega}|\nabla u_{k}|^{p}d\psi-\lambda_{0}\int_{\Omega}|u_{k}|^{\alpha+1}|v_{k}|^{\beta+1}d\psi= o_{k}(1)||u||_{W^{1,p}_{0}(\Omega)},\label{y.5}\\
\int_{\Omega}|\nabla v_{k}|^{q}d\psi-\lambda_{0}\int_{\Omega}|u_{k}|^{\alpha+1}|v_{k}|^{\beta+1}d\psi= o_{k}(1)||v||_{W^{1,q}_{0}(\Omega)}.\label{y.6}
\end{align}
Letting $(\ref{y.4})-(\ref{y.5})\times \frac{\alpha+1}{\gamma_{1}}-(\ref{y.6})\times \frac{\beta+1}{\gamma_{2}}$,
where
\begin{align}
\frac{\gamma_{1}}{\alpha+1}>1, \frac{\gamma_{2}}{\beta+1}>1,\label{y.111}\\
\frac{\alpha+1}{\gamma_{1}}+\frac{\beta+1}{\gamma_{2}}=1,\label{y.11r}
\end{align}
 we have
\begin{align}\label{y.11111r}
(\alpha+1)(\frac{1}{p}-\frac{1}{\gamma_{1}}) ||u_{k}||^{p}_{W^{1,p}_{0}(\Omega)}+ (\beta+1)(\frac{1}{q}-\frac{1}{\gamma_{2}}) ||v_{k}||^{q}_{W^{1,q}_{0}(\Omega)}\leq M ,
\end{align}
 where $M $ is a positive constant.
Since $\alpha+1>p$ and $\beta+1>q$, by (\ref{y.111}) and (\ref{y.11r}) we get that
$\gamma_{1}>p$ and $\gamma_{2}>q$. Thus, by (\ref{y.11111r}) we get that $||u_{k}||_{W^{1,p}_{0}(\Omega)}\leq M_{1}$ and $||v_{k}||_{W^{1,q}_{0}(\Omega)}\leq M_{1}$,
where $M_{1}$ is a positive constant.
By Lemma \ref{qianru} and Remark \ref{remark2}, since $W_{0}^{1,p}(\Omega)$ and $W_{0}^{1,q}(\Omega)$ are pre-compact and $\{(u_{k},v_{k})\}$ is a bounded sequence, we get that $(u_{k},v_{k})\rightarrow (u,v)$ strongly in $\mathbb{W}$. Therefore, $J_{4}(u,v)$ satisfies the $(PS)_{c}$ condition for any $c\in \mathbb{R}$.

Now, we prove that $J_{4}(u,v)$ satisfies all the conditions in the mountain pass theorem in Lemma \ref{mountain}.
Obviously, we have
\begin{align} \label{002}
J_{4}(0,0)=0.
\end{align}
By the Young inequality and the Sobolev embedding in Lemma \ref{qianru}, we get
\begin{align*}
\int_{\Omega}|u|^{\alpha+1}|v|^{\beta+1}d\psi&\leq \frac{\alpha+1}{\gamma_{1}}\int_{\Omega}|u|^{\gamma_{1}}d\psi+\frac{\beta+1}{\gamma_{2}}\int_{\Omega}|v|^{\gamma_{2}}d\psi\\
&\leq \frac{\alpha+1}{\gamma_{1}}C^{\gamma_{1}}_{1}||u||^{\gamma_{1}}_{W^{1,p}_{0}(\Omega)}+\frac{\beta+1}{\gamma_{2}}C^{\gamma_{2}}_{2}||v||^{\gamma_{2}}_{W^{1,q}_{0}(\Omega)},
\end{align*}
where $\alpha,\beta,\gamma_{1}$ and $\gamma_{2}$ satisfy (\ref{y.111}) and (\ref{y.11r}).
Hence, for all $(u,v)\in \mathbb{W}$, we get
\begin{align*}
J_{4}(u,v)\geq&\frac{\alpha+1}{p}||u||^{p}_{W^{1,p}_{0}(\Omega)}+\frac{\beta+1}{q}||v||^{q}_{W^{1,q}_{0}(\Omega)} \notag\\
&-\lambda_{0}\frac{\alpha+1}{\gamma_{1}}C^{\gamma_{1}}_{1}||u||^{\gamma_{1}}_{W^{1,p}_{0}(\Omega)}
-\lambda_{0}\frac{\beta+1}{\gamma_{2}}C^{\gamma_{2}}_{2}||v||^{\gamma_{2}}_{W^{1,q}_{0}(\Omega)}.
\end{align*}
Since $\gamma_{1}>p$ and $\gamma_{2}>q$, we can find some sufficiently small $r>0$ such that
\begin{align} \label{y.10}
 \inf_{||(u,v)||_{\mathbb{W}}=r} J_{4}(u,v)>0.
\end{align}
On the other hand, for any $(u,v)\in \mathbb{W}\backslash \{(0,0)\}$, we have
\begin{align*}
J_{4}(tu,tv)=\frac{\alpha+1}{p} |t|^{p}||u||^{p}_{W^{1,p}_{0}(\Omega)}+\frac{\beta+1}{q}|t|^{q}||v||^{q}_{W^{1,q}_{0}(\Omega)} - \lambda_{0} |t|^{\alpha+\beta+2}\int_{\Omega}|u|^{\alpha+1}|v|^{\beta+1}d\psi.
\end{align*}
Passing to the limit $t\rightarrow +\infty$, since $\alpha+1>p$ and $\beta+1>q$, we get
\begin{align*}
J_{4}(tu,tv) \rightarrow -\infty.
\end{align*}
Hence there exists some $(u_{0},v_{0})\in \mathbb{W}\backslash \{(0,0)\}$ such that
\begin{align} \label{y.13}
J_{4}(u_{0},v_{0}) <0,~when~||(u_{0},v_{0})||_{\mathbb{W}}>r.
\end{align}
By (\ref{002}), (\ref{y.10}) and (\ref{y.13}), we conclude by Lemma \ref{mountain} that
$$c=\inf_{\gamma\in\Gamma}\max_{t\in[0,1]}J_{4}(\gamma(t))$$ is a critical value of $J_{4}$, where
$$\Gamma=\{\gamma\in C([0,1], \mathbb{W}): \gamma(0)=(0,0),\gamma(1)=(u_{0},v_{0})\}.$$
Therefore, there exists some
\begin{align} \label{co2}
(u,v)\in \mathbb{W}\backslash \{(0,0)\}
\end{align}
satisfying
\begin{equation}\label{AE2}
\left \{
\begin{array}{lcr}
 -\Delta_{p} u=\lambda_{0}|u|^{\alpha-1}u|v|^{\beta+1} &{\rm in}&\Omega^{\circ},\\
-\Delta_{q} v=\lambda_{0}|u|^{\alpha+1}|v|^{\beta-1}v &{\rm in}&\Omega^{\circ},\\
u=v=0 &{\rm on}&\partial\Omega.
 \end{array}
\right.
\end{equation}
Since $J_{4}(u,v)=J_{4}(|u|,|v|)$, we may assume that $u\geq 0$ and $v\geq 0$.
Thus, by (\ref{AE2}), we get $-\Delta_{p} u\geq 0 $ and $-\Delta_{q} v\geq 0 $.
Suppose that there exist a point $x_{0}\in \Omega^{\circ}$ such that $u(x_{0})=0=\min_{x\in \Omega^{\circ}}u(x)$, and a point $x_{*}\in \Omega^{\circ}$ such that $v(x_{*})=0=\min_{x\in \Omega^{\circ}}v(x)$.
Then, by Lemma \ref{zhunze1}, we get that $(u,v)\equiv(0,0)$, which is a contradiction with (\ref{co2}).
Therefore, $u> 0, ~v> 0$ for all $x\in \Omega^{\circ}$, and the proof is completed.  ~ ~$\Box$ \\

\noindent $\mathbf{Proof~ of~ Theorem ~\ref{T13}.}$
We define the functional $J_{5}:\mathcal{Z}\rightarrow R$ by
\begin{align}\label{aaa.2}
J_{5}(u,v)=&\frac{1}{p}\int_{\Omega}|\nabla^{m} u|^{p}d\psi+\frac{1}{q}\int_{\Omega}|\nabla^{n} v|^{q}d\psi
- \frac{\lambda}{p}\int_{\Omega}\omega(x)|u|^{p} d\psi-\frac{\vartheta}{q}\int_{\Omega}\psi(x)|v|^{q} d\psi\notag \\
&-\frac{1}{\alpha+\beta+2}\int_{\Omega}|u|^{\alpha+1}|v|^{\beta+1}d\psi.
\end{align}
Firstly, we prove that $J_{5}(u,v)$ satisfies the $(PS)_{c}$ condition for any $c\in \mathbb{R}$.
To see this, we take a sequence $\{(u_{k},v_{k})\}$ in $\mathcal{Z}$ such that as $k\rightarrow +\infty$ we have
\begin{align}
J_{5}(u_{k},v_{k})\rightarrow c, \label{c.2} \\
J_{5}'(u_{k},v_{k})\rightarrow 0.\label{c.3}
\end{align}
By (\ref{canshu111}), (\ref{canshu122}) and (\ref{c.2}), we have
\begin{align} \label{c.5}
\frac{1}{p}(1-\frac{\lambda}{\lambda_{1}})\int_{\Omega}|\nabla^{m} u_{k}|^{p}d\psi+
\frac{1}{q}(1-\frac{\vartheta}{\vartheta_{1}})\int_{\Omega}|\nabla^{n} v_{k}|^{q}d\psi \notag \\
- \frac{1}{\alpha+\beta+2}\int_{\Omega}|u_{k}|^{\alpha+1}|v_{k}|^{\beta+1}d\psi \leq c+o_{k}(1).
\end{align}
By (\ref{c.3}), we have
\begin{align} \label{c.6}
(1-\frac{\lambda}{\lambda_{1}})\int_{\Omega}|\nabla^{m} u_{k}|^{p}d\psi+(1-\frac{\vartheta}{\vartheta_{1}})\int_{\Omega}|\nabla^{n} v_{k}|^{q}d\psi-\int_{\Omega}|u_{k}|^{\alpha+1}|v_{k}|^{\beta+1}d\psi
\leq o_{k}(1).
\end{align}
By (\ref{c.5}) and (\ref{c.6}), we have
\begin{align} \label{c.8}
  (1-\frac{\lambda}{\lambda_{1}}) (\frac{1}{p}-\frac{1}{\alpha+\beta+2}) ||u_{k}||^{p}_{W_{0}^{m,p}(\Omega)}+ (1-\frac{\vartheta}{\vartheta_{1}})(\frac{1}{q}-\frac{1}{\alpha+\beta+2}) ||v_{k}||^{q}_{W_{0}^{n,q}(\Omega)}\leq M,
\end{align}
 where $M$ is a positive constant.
Since $\lambda<\lambda_{1}, \mu<\mu_{1}$ and $\alpha+\beta+2>p,q$, by (\ref{c.8})
we get that $||u_{k}||_{W_{0}^{m,p}(\Omega)}\leq M_{1}$ and $||v_{k}||_{W_{0}^{n,q}(\Omega)}\leq M_{1}$,
where $M_{1}$ is a positive constant.
By Lemma \ref{qianru} and Remark \ref{remark2}, since $W_{0}^{m,p}(\Omega)$ and $W_{0}^{n,q}(\Omega)$ are pre-compact and $\{(u_{k},v_{k})\}$ is a bounded sequence, we get that $(u_{k},v_{k})\rightarrow (u,v)$ strongly in $\mathcal{Z}$. Therefore, $J_{5}(u,v)$ satisfies the $(PS)_{c}$ condition for any $c\in \mathbb{R}$.

Now, we prove that $J_{5}(u,v)$ satisfies all the conditions in the mountain pass theorem in Lemma \ref{mountain}.
Obviously, we have
\begin{align} \label{003}
J_{5}(0,0)=0.
\end{align}
By the Young inequality and the Sobolev embedding in Lemma \ref{qianru}, we get
\begin{align} \label{c.9}
\int_{\Omega}|u|^{\alpha+1}|v|^{\beta+1}d\psi&\leq \frac{\alpha+1}{\gamma_{1}}\int_{\Omega}|u|^{\gamma_{1}}d\psi+\frac{\beta+1}{\gamma_{2}}\int_{\Omega}|v|^{\gamma_{2}}d\psi \notag\\
&\leq\frac{\alpha+1}{\gamma_{1}}C^{\gamma_{1}}_{3}||u||^{\gamma_{1}}_{W_{0}^{m,p}(\Omega)}
+\frac{\beta+1}{\gamma_{2}}C^{\gamma_{2}}_{4}||v||^{\gamma_{2}}_{W_{0}^{n,q}(\Omega)},
\end{align}
where $\alpha,\beta,\gamma_{1}$ and $\gamma_{2}$ satisfy the following two conditions
\begin{align}
\frac{\gamma_{1}}{\alpha+1}>1, \frac{\gamma_{2}}{\beta+1}>1, \label{c.10}\\
\frac{\alpha+1}{\gamma_{1}}+\frac{\beta+1}{\gamma_{2}}=1. \label{c.11}
\end{align}
Hence, for all $(u,v)\in \mathcal{Z}$, by (\ref{aaa.2}) and (\ref{c.9}) we get
\begin{align} \label{c.16}
J_{5}(u,v)\geq&\frac{1}{p}(1-\frac{\lambda}{\lambda_{1}})||u||^{p}_{W_{0}^{m,p}(\Omega)}+\frac{1}{q}(1-\frac{\vartheta}{\vartheta_{1}})||v||^{q}_{W_{0}^{n,q}(\Omega)}
-\frac{\alpha+1}{\gamma_{1}(\alpha+\beta+2)}\cdot C^{\gamma_{1}}_{1}||u||^{\gamma_{1}}_{W_{0}^{m,p}(\Omega)}\notag\\
&-\frac{1}{\alpha+\beta+2}\cdot\frac{\beta+1}{\gamma_{2}}\cdot C^{\gamma_{2}}_{2}||v||^{\gamma_{2}}_{W_{0}^{n,q}(\Omega)}.
\end{align}
Since $\alpha+1>p$ and $\beta+1>q$, by (\ref{c.10}) and (\ref{c.11}), we get that $\gamma_{1}>p$ and $\gamma_{2}>q$.
Thus, by (\ref{c.16}), we can find some sufficiently small $r>0$ such that
\begin{align} \label{c.17}
 \inf_{||(u,v)||_{\mathcal{Z}}=r} J_{5}(u,v)>0.
\end{align}
On the other hand, there exists $(u,v)\in \mathcal{Z}\backslash \{(0,0)\}$ such that
\begin{align} \label{c.20}
J_{5}(tu,tv)=&\frac{t^{p}}{p}(1-\frac{\lambda}{\lambda_{1}})||u||^{p}_{W_{0}^{m,p}(\Omega)}+\frac{t^{q}}{q}(1-\frac{\vartheta}{\vartheta_{1}})||v||^{q}_{W_{0}^{n,q}(\Omega)} \notag\\
&-\frac{t^{\alpha+\beta+2}}{\alpha+\beta+2}\int_{\Omega}|u|^{\alpha+1}|v|^{\beta+1}d\psi.
\end{align}
Passing to the limit $t\rightarrow +\infty$, since $\alpha+1>p$ and $\beta+1>q$, by (\ref{c.20}) we get
\begin{align*}
J_{5}(tu,tv) \rightarrow -\infty.
\end{align*}
Hence there exists some $(u_{0},v_{0})\in \mathcal{Z}\backslash \{(0,0)\}$ such that
\begin{align} \label{c.23}
J_{5}(u_{0},v_{0}) <0,~when~||(u_{0},v_{0})||_{\mathcal{Z}}>r.
\end{align}
By (\ref{003}), (\ref{c.17}) and (\ref{c.23}), we conclude by Lemma \ref{mountain} that
$$c=\inf_{\gamma\in\Gamma}\max_{t\in[0,1]}J_{5}(\gamma(t))$$ is a critical value of $J_{5}$, where
$$\Gamma=\{\gamma\in C([0,1], \mathcal{Z}): \gamma(0)=(0,0),\gamma(1)=(u_{0},v_{0})\}.$$
Thus, we complete this proof. ~$\Box$

\section{Proof of the global existence results on finite graphs}

In this section, by using the mountain pass theorem and some Lemmas in section \ref{Preliminary}, we prove the global existence results
(Theorem \ref{T3}-\ref{T23}) on finite graphs.

\noindent $\mathbf{Proof~ of~ Theorem ~\ref{T3}.}$
Now, we define the functional $J_{6}:\mathbb{L}\rightarrow
\mathbb{R}$ by
\begin{align} \label{6.1}
J_{6}(u,v)=\frac{1}{2}||(u,v)||^{2}_{\mathbb{L}}-\int_{V}F(x,u^{+})d\psi-\int_{V}G(x,v^{+})d\psi,
\end{align}
where $u^{+}=\max\{u(x),0\}$ and $v^{+}=\max\{v(x),0\}$.
Firstly, we prove that $J_{6}(u,v)$ satisfies the $(PS)_{c}$ condition for any $c\in \mathbb{R}$.
To see this, we take a sequence $\{(u_{k},v_{k})\}$ in $\mathbb{L}$ such that $J_{6}(u_{k},v_{k})\rightarrow c$ and $J_{6}'(u_{k},v_{k})\rightarrow 0$
as $k\rightarrow +\infty$. Thus, we have
\begin{align}
\frac{1}{2}||(u_{k},v_{k})||^{2}_{\mathbb{L}}-\int_{V}F(x,u^{+}_{k})d\psi-\int_{V}G(x,v^{+}_{k})d\psi= c+o_{k}(1)||(u_{k},v_{k})||_{\mathbb{L}},\label{6.11}\\
||(u_{k},v_{k})||^{2}_{\mathbb{L}}-\int_{V}u^{+}_{k}f(x,u^{+}_{k})d\psi-\int_{V}v^{+}_{k}g(x,v^{+}_{k})d\psi =o_{k}(1)||(u_{k},v_{k})||_{\mathbb{L}}.\label{6.12}
\end{align}
By $(H4)$, (\ref{6.11}) and (\ref{6.12}), we have
\begin{align*}
(\frac{1}{2}-\frac{1}{\theta})||(u_{k},v_{k})||^{2}_{\mathbb{L}}\leq M,
\end{align*}
where $M$ is a positive constant.
Since $\theta>2$, we get that $\{(u_{k},v_{k})\}$ is bounded in $\mathbb{L}$.
By Lemma \ref{qianru2} and Remark \ref{remark2}, since $W^{1,2}(V)$ is pre-compact and $\{(u_{k},v_{k})\}$ is a bounded sequence, we get that $(u_{k},v_{k})\rightarrow (u,v)$ strongly in $\mathbb{L}$. Therefore, $J_{6}(u,v)$ satisfies the $(PS)_{c}$ condition for any $c\in \mathbb{R}$.

Now, we prove that $J_{6}(u,v)$ satisfies all the conditions in the mountain pass theorem in Lemma \ref{mountain}.\\
\indent $\mathbf{Step~ 1.}$ Obviously, we have
\begin{align} \label{004}
J_{6}(0,0)=0.
\end{align}
\indent $\mathbf{Step~ 2.}$ Firstly, we get some inequalities regarding $F(x,u^{+})$ and $G(x,v^{+})$.\\
 By $(H3)$, for any $\varepsilon>0$, there exists $\delta=\delta(\varepsilon)>0$, such that
\begin{align*}
|f(x,u^{+})|\leq\varepsilon |u|, whenever~~ |u|\leq \delta.
\end{align*}
By $(H2)$, there exists $K>0$, such that
\begin{align*}
|f(x,u^{+})|\leq |u|^{s}, whenever~~ |u|\geq K.
\end{align*}
By $(H1)$, there exists $b=b(\varepsilon)>0$, such that
\begin{align*}
|f(x,u^{+})|\leq b\leq \frac{b|u|^{s}}{\delta^{s}}= c_{1}(\varepsilon) |u|^{s}, whenever~~\delta \leq|u|\leq K.
\end{align*}
Thus, we get
\begin{align} \label{6.17}
|f(x,u^{+})|\leq \varepsilon|u| +c_{2}(\varepsilon) |u|^{s}, \forall (x,u)\in V\times \mathbb{R},
\end{align}
where $c_{2}(\varepsilon)=c_{1}(\varepsilon)+1$. Therefore, by (\ref{6.17}) we have
\begin{align*}
F(x,u^{+})\leq & \int^{u^{+}}_{0}(\varepsilon|t| +c_{2}(\varepsilon) |t|^{s})dt,\notag \\
\leq&  \frac{\varepsilon}{2}|u|^{2}+c(\varepsilon) |u|^{s+1},  \forall (x,u)\in V\times \mathbb{R},
\end{align*}
where $c(\varepsilon)=\frac{c_{2}(\varepsilon)}{s+1}$.
Similarly, we can get
\begin{align*}
G(x,v^{+})\leq \frac{\varepsilon}{2}|v|^{2}+c(\varepsilon) |v|^{s+1},  \forall (x,v)\in V\times \mathbb{R}.
\end{align*}
On the other hand, by $(H4)$, there exist two positive constants $c_{3}$ and $c_{4}$, such that
\begin{align}
F(x,u^{+})\geq c_{3} |u|^{\theta}-c_{4},  \forall (x,u)\in V\times \mathbb{R},\label{6.20}\\
G(x,v^{+})\geq c_{3} |v|^{\theta}-c_{4},  \forall (x,v)\in V\times \mathbb{R}.\label{6.21}
\end{align}
\indent $\mathbf{Step ~3.}$ Verify the geometric conditions (two inequalities) of the Mountain Pass Lemma.

By (\ref{6.20}), (\ref{6.21}) and Lemma \ref{qianru2}, for all $(u,v)\in \mathbb{L}$, we get
\begin{align*}
J_{6}(u,v)\geq&\frac{1}{2}||(u,v)||^{2}_{\mathbb{L}}-\int_{V}(\frac{\varepsilon}{2}|u|^{2}+c(\varepsilon) |u|^{s+1})d\psi
-\int_{V}(\frac{\varepsilon}{2}|v|^{2}+c(\varepsilon) |v|^{s+1})d\psi \notag \\
\geq&\frac{1}{2}(1-a^{2}_{1}\varepsilon)||(u,v)||^{2}_{\mathbb{L}}-a^{s+1}_{2}c(\varepsilon)||(u,v)||^{s+1}_{\mathbb{L}},
\end{align*}
where $a_{1}$ and $a_{2}$ are two positive constants.
Since $s+1>2$, we can find some sufficiently small $\varepsilon>0$ and $r>0$ such that
\begin{align} \label{6.23}
 \inf_{||(u,v)||_{\mathbb{L}}=r} J_{6}(u,v)>0.
\end{align}
On the other hand, for any $(u,v)\in \mathbb{L}\backslash \{(0,0)\}$, by (\ref{6.20}) and (\ref{6.21}), we have
\begin{align*}
J_{6}(u,v)& \leq \frac{1}{2}||(u,v)||^{2}_{\mathbb{L}}-\int_{V}(c_{3} |u|^{\theta}-c_{4})d\psi -\int_{V}(c_{3} |v|^{\theta}-c_{4})d\psi \notag \\
&=\frac{1}{2}||(u,v)||^{2}_{\mathbb{L}}-c_{3}\int_{V}( |u|^{\theta}+|v|^{\theta})d\psi +2c_{4}|V|.
\end{align*}
For any given $(u,v)\in \mathbb{L}\backslash \{(0,0)\}$, since $\theta>2$, we have
\begin{align*}
J_{6}(tu,tv)\leq\frac{t^{2}}{2}||(u,v)||^{2}_{\mathbb{L}}-c_{3}t^{\theta}\int_{V}( |u|^{\theta}+|v|^{\theta})d\psi +2c_{4}|V| \rightarrow -\infty, as ~~t\rightarrow +\infty.
\end{align*}
Hence there exists some $(u_{0},v_{0})\in \mathbb{L}\backslash \{(0,0)\}$ such that
\begin{align} \label{6.25}
J_{6}(u_{0},v_{0}) <0,~when ~ ||(u_{0},v_{0})||_{\mathbb{L}}>r.
\end{align}
By (\ref{004}), (\ref{6.23}) and (\ref{6.25}), we conclude by Lemma \ref{mountain} that
$$c=\inf_{\gamma\in\Gamma}\max_{t\in[0,1]}J_{6}(\gamma(t))$$ is a critical value of $J_{6}$, where
$$\Gamma=\{\gamma\in C([0,1], \mathbb{L}): \gamma(0)=(0,0),\gamma(1)=(u_{0},v_{0})\}.$$
Therefore, there exists some
\begin{align} \label{co1}
(u,v)\in \mathbb{L}\backslash \{(0,0)\}
\end{align}
satisfying
\begin{equation}\label{111111222}
\left \{
\begin{array}{lcr}
 -\Delta u+h(x)u=g(x,v^{+}) &{\rm in}& V,\\
-\Delta v+h(x)v=f(x,u^{+}) &{\rm in}& V.
 \end{array}
\right.
\end{equation}
Set $u^{-}=\min\{u(x),0\}$. It is easy to see that $u=u^{+}+u^{-}$.
Noting that
$$|\nabla u^{-}|^{2}\leq \Gamma(u^{-})+\Gamma(u^{-},u^{+})=\Gamma(u^{-},u),$$
 and $g(x,v^{+})\geq0$, by the first equation in (\ref{111111222}), we get
\begin{align*}
\int_{V}(|\nabla u^{-}|^{2}+h|u^{-}|^{2})d\psi\leq& -\int_{V}u^{-}\Delta ud\psi+\int_{V}hu^{-}ud\psi \notag \\
=&\int_{V}u^{-}g(x,v^{+})d\psi\leq0.
\end{align*}
So we get $u^{-}\equiv 0$, and thus $u=u^{+}+u^{-}\geq 0$ for all $x\in V$.
Set $v^{-}=\min\{v(x),0\}$. Similarly, by the second equation in (\ref{111111222}), we have that $v^{-}\equiv 0$ and $v=v^{+}+v^{-}\geq 0$ for all $x\in V$.
Therefore, we get that
\begin{align} \label{LLco3}
(u,v)\in \mathbb{L}\backslash \{(0,0)\}
\end{align}
satisfies
\begin{equation}\label{E3555544}
\left \{
\begin{array}{lcr}
 -\Delta u+h(x)u=g(x,v) &{\rm in}& V,\\
-\Delta v+h(x)v=f(x,u) &{\rm in}& V,
 \end{array}
\right.
\end{equation}
where $u(x)\geq 0$ and $v(x)\geq0$ for all $x\in V$.
Since $g(x,v^{+})\geq 0$ and $f(x,u^{+})\geq 0$, we get
 \begin{equation*}
\left \{
\begin{array}{lcr}
 -\Delta u+h(x)u\geq 0\\
-\Delta v+h(x)v\geq 0.
 \end{array}
\right.
\end{equation*}
Suppose that there exist a point $x_{0}\in V$ such that $u(x_{0})=0=\min_{x\in V}u(x)$, and a point $x_{*}\in V$ such that $v(x_{*})=0=\min_{x\in V}v(x)$.
Then, by Lemma \ref{zhunze1}, we get $(u,v)\equiv(0,0)$, which is a contradiction with (\ref{LLco3}).
 Therefore, $u(x)> 0$ and $v(x)>0$ for all $x\in V$. Thus, the proof is completed.  ~ ~$\Box$ \\

\noindent $\mathbf{Proof~ of~ Theorem ~\ref{T4}.}$
Now, we define the functional $J_{7}:\mathbb{X}\rightarrow \mathbb{R}$ by
\begin{align} \label{9.1}
J_{7}(u,v)=&\frac{1}{p}\int_{V}(|\nabla u|^{p}+h(x)|u|^{p})d\psi +\frac{1}{q}\int_{V}(|\nabla v|^{q}+h(x)|v|^{q})d\psi \notag \\
&-\int_{V}F(x,u^{+},v^{+})d\psi,
\end{align}
where $u^{+}=\max\{u(x),0\}$ and $v^{+}=\max\{v(x),0\}$.
Firstly, we prove that $J_{7}(u,v)$ satisfies the $(PS)_{c}$ condition for any $c\in \mathbb{R}$.
To see this, we take a sequence $\{(u_{k},v_{k})\}$ in $\mathbb{X}$ such that $J_{7}(u_{k},v_{k})\rightarrow c$ and $J_{7}'(u_{k},v_{k})\rightarrow 0$
as $k\rightarrow +\infty$. Thus, we have
\begin{align}
c+o_{k}(1)\geq -\int_{V}\big{(}\theta_{1}u^{+}_{k}F_{u}(x,u^{+}_{k},v^{+}_{k})+\theta_{2}v^{+}_{k}F_{v}(x,u^{+}_{k},v^{+}_{k})\big{)}d\psi \notag \\
 +\frac{1}{p}\int_{V}(|\nabla u_{k}|^{p}+h(x)|u_{k}|^{p})d\psi +\frac{1}{q}\int_{V}(|\nabla v_{k}|^{q}+h(x)|v_{k}|^{q})d\psi ; \label{9.11}\\
o_{k}(1)=\int_{V}(|\nabla u_{k}|^{p}+h(x)|u_{k}|^{p})d\psi  -\int_{V}u^{+}_{k}F_{u}(x,u^{+}_{k},v^{+}_{k})d\psi ;\label{9.12}\\
o_{k}(1)=\int_{V}(|\nabla v_{k}|^{q}+h(x)|v_{k}|^{q})d\psi  -\int_{V}v^{+}_{k}F_{v}(x,u^{+}_{k},v^{+}_{k})d\psi .\label{9.13}
\end{align}
By (\ref{9.11}), (\ref{9.12}) and (\ref{9.13}), we have
\begin{align*}
(\frac{1}{p}-\theta_{1})||u_{k}||^{p}_{W^{1,p}(V)} + (\frac{1}{q}-\theta_{2})||v_{k}||^{q}_{W^{1,q}(V)} \leq M,
\end{align*}
where $M$ is a positive constant.
Since $0<\theta_{1}<\frac{1}{p}$ and $0<\theta_{2}<\frac{1}{q}$, we get $\{(u_{k},v_{k})\}$ is bounded in $\mathbb{X}$.
By Lemma \ref{qianru2} and Remark \ref{remark2}, since $W^{1,p}(V)$ and $W^{1,q}(V)$ are pre-compact and $\{(u_{k},v_{k})\}$ is a bounded sequence, we get that $(u_{k},v_{k})\rightarrow (u,v)$ strongly in $\mathbb{X}$. Therefore, $J_{7}(u,v)$ satisfies the $(PS)_{c}$ condition for any $c\in \mathbb{R}$.

Now, we prove that $J_{7}(u,v)$ satisfies all the conditions in the mountain pass theorem in Lemma \ref{mountain}.
Obviously, we have
\begin{align} \label{005}
J_{7}(0,0)=0.
\end{align}
Since $|\nabla u^{+}(x)|\leq |\nabla u(x)|$,
by (\ref{9.1}), $(H2)$, $(H3)$ and Lemma \ref{qianru2}, for all $(u,v)\in \mathbb{X}$, we get
\begin{align*}
J_{7}(u,v)\geq&-C_{1}(||u||^{r}_{W^{1,p}(V)}+ ||v||^{s}_{W^{1,q}(V)}+||u||^{\bar{r}}_{W^{1,p}(V)}+||v||^{\bar{s}}_{W^{1,q}(V)}) \notag\\
   &+\frac{1}{p}||u||^{p}_{W^{1,p}(V)}+\frac{1}{q}||v||^{q}_{W^{1,q}(V)}.
\end{align*}
Thus, we can find some sufficiently small $\rho>0$ such that when $||(u,v)||_{\mathbb{X}}=\rho$, we have
\begin{align} \label{11.3}
 J_{7}(u,v)>0.
\end{align}
On the other hand, by $(H_{4})$, we have
\begin{align*}
F\big{(}x,(u^{+})^{t^{\theta_{1}}},(v^{+})^{t^{\theta_{2}}}\big{)}>tC_{2},
\end{align*}
where $C_{2}$ is a positive constant.
 For any $(u,v)\in \mathbb{X}\backslash \{(0,0)\}$, since $0<\theta_{1}p<1$ and $0<\theta_{2}q<1$, we have
\begin{align*}
J_{7}(tu,tv)\leq \frac{1}{p}t^{\theta_{1}p}||u||^{p}_{W^{1,p}(V)}+\frac{1}{q}t^{\theta_{2}q}||v||^{q}_{W^{1,q}(V)}-tC_{2}|V|\rightarrow -\infty, as ~~t\rightarrow +\infty.
\end{align*}
Hence there exists some $(u_{0},v_{0})\in \mathbb{X}\backslash \{(0,0)\}$ such that
\begin{align} \label{11.8}
J_{7}(u_{0},v_{0}) <0,~when~ ||u||^{p}_{W^{1,p}}(V)>r ~and~||v||^{q}_{W^{1,q}}(V)>r.
\end{align}
By (\ref{005}), (\ref{11.3}) and (\ref{11.8}), we conclude by Lemma \ref{mountain} that
$$c=\inf_{\gamma\in\Gamma}\max_{t\in[0,1]}J_{7}(\gamma(t))$$ is a critical value of $J_{7}$, where
$$\Gamma=\{\gamma\in C([0,1], \mathbb{X}): \gamma(0)=(0,0),\gamma(1)=(u_{0},v_{0})\}.$$
Therefore, there exists some
\begin{align*}
(u,v)\in \mathbb{X}\backslash \{(0,0)\}
\end{align*}
satisfying
\begin{equation}\label{AAA}
\left \{
\begin{array}{lcr}
 -\Delta_{p} u+h(x)|u|^{p-2}u=f(x,u^{+},v^{+})  &{\rm in}& V,\\
-\Delta_{q} v+h(x)|v|^{q-2}v=g(x,u^{+},v^{+})  &{\rm in}& V.
 \end{array}
\right.
\end{equation}
Set $u^{-}=\min\{u(x),0\}$. Obviously, $u=u^{+}+u^{-}$.
Noting that
$$|\nabla u^{-}|^{2}\leq \Gamma(u^{-})+\Gamma(u^{-},u^{+})=\Gamma(u^{-},u)$$
 and $f(x,0,0)=0$, by the first equation in (\ref{AAA}), we get
\begin{align*}
\int_{V}(|\nabla u^{-}|^{p}+h|u^{-}|^{p})d\psi\leq& -\int_{V}u^{-}\Delta_{p}ud\psi+\int_{V}hu^{-}|u^{-}|^{p-2}u d\psi \notag \\
=&\int_{V}u^{-}f(x,u^{+},v^{+})d\psi=0.
\end{align*}
So we get $u^{-}\equiv 0$, and thus $u=u^{+}+u^{-}\geq 0$ for all $x\in V$.
Set $v^{-}=\min\{v(x),0\}$. Similarly, by the second equation in (\ref{AAA}), we have $v^{-}\equiv 0$ and $v=v^{+}+v^{-}\geq 0$ for all $x\in V$.
Therefore, we get that
\begin{align} \label{co3}
(u,v)\in \mathbb{X}\backslash \{(0,0)\}
\end{align}
satisfies
\begin{equation*}
\left \{
\begin{array}{lcr}
 -\Delta_{p} u+h(x)|u|^{p-2}u=f(x,u,v)  &{\rm in}& V,\\
-\Delta_{q} v+h(x)|v|^{q-2}v=g(x,u,v)  &{\rm in}& V,
 \end{array}
\right.
\end{equation*}
where $u(x)\geq 0$ and $v(x)\geq0$ for all $x\in V$.
Suppose that there exist a point $x_{0}\in V$ such that $u(x_{0})=0=\min_{x\in V}u(x)$, and a point $x_{*}\in V$ such that $v(x_{*})=0=\min_{x\in V}v(x)$.
Then, by $(H1)$ and Lemma \ref{zhunze1}, we get $(u,v)\equiv(0,0)$, which is a contradiction with (\ref{co3}).
Therefore, $u(x)> 0$ and $v(x)> 0$ for all $x\in V$, and the proof is completed.  ~ ~$\Box$ \\

\noindent $\mathbf{Proof~ of~ Theorem ~\ref{T23}.}$
Now, we define the functional $J_{vmn}:\mathbb{Q}\rightarrow R$ by
\begin{align*}
J_{vmn}(u,v)=&\frac{1}{p}\int_{V}(|\nabla^{m} u|^{p}+h|u|^{p})d\psi+\frac{1}{q}\int_{V}(|\nabla^{n} v|^{q}+h|v|^{q})d\psi \notag\\
&-\int_{V}F(x,u)d\psi-\int_{V}G(x,v)d\psi.
\end{align*}
Firstly, we prove that $J_{vmn}(u,v)$ satisfies the $(PS)_{c}$ condition for any $c\in \mathbb{R}$.
To see this, we take a sequence $\{(u_{k},v_{k})\}$ in $\mathbb{Q}$ such that
 $J_{vmn}(u_{k},v_{k})\rightarrow c$ and $J'_{vmn}(u_{k},v_{k})\rightarrow 0$
as $k\rightarrow +\infty$. Thus, we have
\begin{align}
\frac{1}{p}\int_{V}(|\nabla^{m} u_{k}|^{p}&+h|u_{k}|^{p})d\psi +\frac{1}{q}\int_{V}(|\nabla^{n} v_{k}|^{q}+h|v_{k}|^{q})d\psi -\int_{V}F(x,u_{k})d\psi\notag\\
&-\int_{V}G(x,v_{k})d\psi= c+o_{k}(1)||(u_{k},v_{k})||_{\mathbb{Q}},\label{23.2}
\end{align}
and
\begin{align}
\int_{V}(|\nabla^{m} u_{k}|^{p}+h|u_{k}|^{p})d\psi-\int_{V}u_{k}f(x,u_{k})d\psi =o_{k}(1), \label{23.12}\\
\int_{V}(|\nabla^{n} v_{k}|^{q}+h|v_{k}|^{q})d\psi-\int_{V}v_{k}g(x,v_{k})d\psi =o_{k}(1).\label{23.13}
\end{align}
By $(H3)$, (\ref{23.2}), (\ref{23.12}) and (\ref{23.13}), we have
\begin{align*}
(\frac{1}{p}-\frac{1}{\theta_{0}})||u_{k}||^{p}_{W^{m,p}(V)}+(\frac{1}{q}-\frac{1}{\theta_{0}})||v_{k}||^{q}_{W^{n,q}(V)}\leq M,
\end{align*}
where $M$ is a positive constant.
Since $\theta_{0}>p,q$, we get $\{(u_{k},v_{k})\}$ is bounded in $\mathbb{Q}$.
By Lemma \ref{qianru2} and Remark \ref{remark2}, since $W^{m,p}(V)$ and $W^{n,q}(V)$ are pre-compact and $\{(u_{k},v_{k})\}$ is a bounded sequence, we get that $(u_{k},v_{k})\rightarrow (u,v)$ strongly in $\mathbb{Q}$. Therefore, $J_{vmn}(u,v)$ satisfies the $(PS)_{c}$ condition for any $c\in \mathbb{R}$.

Now, we prove that $J_{vmn}(u,v)$ satisfies all the conditions in the mountain pass theorem in Lemma \ref{mountain}.
 Obviously, we have
 \begin{align} \label{006}
 J(0,0)=0.
 \end{align}
From \cite{AGY}, we get some inequalities regarding $F(x,u)$ and $G(x,v)$.
By $(H2)$, for all $t\in \mathbb{R}$ and $x\in V$, there exist some $\lambda<\lambda_{mp}(V)$ and $\delta>0$ such that
\begin{align*}
F(x,t)\leq\frac{\lambda}{p}|t|^{p}+\frac{|t|^{p+1}}{\delta^{p+1}}|F(x,t)|.
\end{align*}
For any $u\in W^{m,p}(V)$ with $||u||_{W^{m,p}(V)}\leq 1$, we have $||u||_{\infty}\leq C_{0}$ and thus $||F(x,u)||_{\infty}\leq C_{0}$.
Hence
\begin{align*}
\int_{V}F(x,t)\leq&\frac{\lambda}{p}\int_{V}|u|^{p}d\psi+\frac{C_{0}}{\delta^{p+1}}\int_{V}|u|^{p+1}|F(x,t)|d\psi\notag\\
\leq&\frac{\lambda}{\lambda_{mp}(V)p}||u||^{p}_{W^{m,p}(V)} +C_{0}||u||^{p+1}_{W^{m,p}(V)}.
\end{align*}
Similarly, by $(H2)$, for all $t\in \mathbb{R}$ and $x\in V$, there exist some $\lambda_{0}<\lambda_{nq}(V)$ and $\delta_{0}>0$ such that
\begin{align*}
G(x,t)\leq\frac{\lambda_{0}}{q}|t|^{q}+\frac{|t|^{q+1}}{\delta_{0}^{q+1}}|G(x,t)|.
\end{align*}
For any $v\in W^{n,q}(V)$ with $||v||_{W^{n,q}(V)}\leq 1$, we have $||v||_{\infty}\leq C_{01}$ and thus $||G(x,v)||_{\infty}\leq C_{01}$.
Hence
\begin{align*}
\int_{V}G(x,v)\leq&\frac{\lambda_{0}}{q}\int_{V}|v|^{q}d\psi+\frac{C_{01}}{\delta_{0}^{q+1}}\int_{V}|v|^{q+1}|G(x,v)|d\psi\notag\\
\leq&\frac{\lambda_{0}}{\lambda_{nq}(V)q}||u||^{q}_{W^{n,q}(V)} +C_{01}||v||^{q+1}_{W^{n,q}(V)}.
\end{align*}
It follows that
\begin{align*}
J_{vmn}(u,v)\geq&\frac{\lambda_{mp}(V)-\lambda}{\lambda_{mp}(V)p}||u||^{p}_{W^{m,p}(V)}+\frac{\lambda_{nq}(V)-\lambda_{0}}{\lambda_{nq}(V)q}||u||^{q}_{W^{n,q}(V)}\notag\\
&-C_{0}||u||^{p+1}_{W^{m,p}(V)}-C_{01}||v||^{q+1}_{W^{n,q}(V)}.
\end{align*}
Therefore, we can find some sufficiently small $r>0$ such that
\begin{align} \label{23.59}
 \inf_{||u||_{\mathbb{Q}}=r} J(u,v)>0.
\end{align}
On the other hand, by $(H3)$, there exist some positive constants $\xi_{1},\xi_{2},\xi_{3}$ and $\xi_{4}$ such that
\begin{align*}
F(x,s)\geq \xi_{1}|s|^{\theta_{0}}-\xi_{2},~~\forall s\in \mathbb{R},\\
G(x,s)\geq \xi_{3}|s|^{\theta_{0}}-\xi_{4},~~\forall s\in \mathbb{R}.
\end{align*}
For any given $(u,v)\in \mathbb{Q}\backslash \{(0,0)\}$, since $\theta_{0}>p,q$, we have
\begin{align*}
J_{vmn}(tu,tv)\leq&\frac{t^{p}}{p}||u||^{p}_{W^{m,p}(V)}+\frac{t^{q}}{q}||u||^{q}_{W^{n,q}(V)}
- \xi_{1}|t|^{\theta_{0}}\int_{V}|u|^{\theta_{0}}d\psi+\xi_{2}|V| \notag\\
& - \xi_{3}|t|^{\theta_{0}}\int_{V}|v|^{\theta_{0}}d\psi+\xi_{4}|V| \rightarrow -\infty, as ~~t\rightarrow +\infty.
\end{align*}
Hence there exists some $(u_{0},v_{0})\in \mathbb{Q}\backslash \{(0,0)\}$ such that
\begin{align} \label{23.90}
J(u_{0},v_{0}) <0,~when~||(u_{0},v_{0})||_{\mathbb{Q}}>r.
\end{align}
By (\ref{006}), (\ref{23.59}) and (\ref{23.90}), we conclude by Lemma \ref{mountain} that
$$c=\inf_{\gamma\in\Gamma}\max_{t\in[0,1]}J_{vmn}(\gamma(t))$$ is a critical value of $J_{vmn}$, where
$$\Gamma=\{\gamma\in C([0,1], \mathbb{Q}): \gamma(0)=(0,0),\gamma(1)=(u_{0},v_{0})\}.$$
Thus, we complete this proof.      ~$\Box$

\end{document}